\documentclass[10pt]{amsart}

\usepackage{amsmath}
\usepackage{amssymb}
\usepackage[shortalphabetic]{amsrefs}
\usepackage{stmaryrd}
\usepackage{verbatim}
\usepackage{enumerate}
\usepackage[all]{xy}
\usepackage{tikz}

\numberwithin{equation}{section}
\newtheorem{thm}[equation]{Theorem}
\newtheorem*{thm*}{Theorem}

\newtheorem{prop}[equation]{Proposition}

\newtheorem{ex}[equation]{Example}
\newtheorem{coro}[equation]{Corollary}

\theoremstyle{remark}
\newtheorem*{remark*}{Remark}
\newtheorem{remark}[equation]{Remark}

\theoremstyle{definition}

\numberwithin{figure}{section}
\numberwithin{table}{section}

\newcommand{\cev}[1]{\reflectbox{\ensuremath{\vec{\reflectbox{\ensuremath{#1}}}}}}

\DeclareMathOperator{\End}{End}

\DeclareMathOperator{\Adj}{Adj }
\DeclareMathOperator{\Der}{Der }

\DeclareMathOperator{\GL}{GL}

\DeclareMathOperator{\Sp}{Sp}

\DeclareMathOperator{\Gal}{Gal }

\DeclareMathOperator{\gl}{\mathfrak{gl}}
\DeclareMathOperator{\Inn}{Inn}
\DeclareMathOperator{\Aut}{Aut}

\DeclareMathOperator{\Isom}{Isom}

\DeclareMathOperator{\im}{im }

\DeclareMathOperator{\Out}{Out}

\DeclareMathOperator{\ad}{ad}

\newcommand{\M}{\mathbb{M}}
\newcommand{\cond}[2]{\overset{#1}{{_{#1}#2}_{#1}}}

\newcommand{\CC}[1]{\mathcal{C}_{#1}}
\newcommand{\LL}[1]{\mathcal{L}_{#1}}
\newcommand{\RR}[1]{\mathcal{R}_{#1}}
\newcommand{\MM}[1]{\mathcal{M}_{#1}}
\newcommand{\LMR}{\mathcal{LMR}}

\newcommand{\lversor}{\reflectbox{\ensuremath{\oslash}}}
\newcommand{\rversor}{\oslash}

\newcommand{\bmto}{\rightarrowtail}

\title{Skolem-Noether for nilpotent products}

\author{James B. Wilson}
\address{
	Department of Mathematics\\
	Colorado State University\\
	Fort Collins, CO 80523\\
}
\email{James.Wilson@ColoState.Edu}
\date{\today}

\begin{document}

\maketitle

\begin{abstract}
We consider the structure of groups and algebras that can be represented as automorphisms,
respectively derivations, of bilinear maps.  This clarifies many features found in automorphisms
of associative and Lie algebras, and of groups that have nontrivial nilpotent radicals.
We introduce  fundamental structures and prove results akin to
theorems of Morita and Skolem-Noether.  
Applications and examples are included.
\end{abstract}

\section{Introduction}

When studying $M$-filtered algebras $A=\bigcup_{s\in M} A_s$
one inevitably encounters associated $k$-bimaps ($k$-bilinear maps). 
Here and throughout $M=\langle M, \prec, +,0\rangle$ is a pre-ordered 
commutative monoid, e.g. $\mathbb{N}^c$, and filters require that 
$A_s \cdot A_t\leq A_{s+t}$ and that $s\prec t$ implies $A_s\geq A_t$.  
The appropriate analog for groups $G=\bigcup_{s\in M} G_s$ replaces 
the product with group commutators $[x,y]=x^{-1} y^{-1} xy$, for instance 
the lower central series is a filter but there are many more.
Setting $\partial A_s=\langle A_{s+t}: t\neq 0\rangle$ allows for the restriction 
of the product in $A$ to $A_s/\partial A_{s}\times A_t/\partial A_t\bmto A_{s+t}/\partial A_{s+t}$,
and these are the bimaps that most commonly arise.  Details and applications can be found in several sources
including \citelist{\cite{Rowen:I}*{Chapter 7}\cite{Khukhro}\cite{Wilson:alpha}}.  

Properties of associated bimaps transfer naturally to the original filtered groups and algebras.
For instance a decomposition of a bimap into pairwise orthogonal factors characterizes the direct
and central products of groups and algebras \citelist{\cite{Myasnikov}\cite{Wilson:unique-cent}\cite{Wilson:RemakI}}.  
Isomorphisms and courser equivalences of isologism and isotopism translate to groups acting on bimaps.
Several recent projects are making bimaps a subject in their own right and these ideas offer valuable context
and techniques; cf.
\citelist{\cite{BFFM}\cite{First:forms}\cite{First:Morita}\cite{LW}\cite{Myasnikov}\cite{Wilson:division}\cite{BW:autotopism}}.

Our objective is to ascribe new structure to bimaps.  We are particularly interested in what groups and 
algebras can act on a bimap.  To explain this we introduce
a notion of ``inner/outer'' action for bimaps, and prove theorems of Skolem-Noether and Morita type.  These
remove many complexities concerning automorphism groups of nilpotent groups, and algebras and general filtered
products. 

Evidently theorems from Ring Theory inspired this work.
Yet, before we proceed a caution seems necessary.  Our subject is bimaps.
Take for example the mundane bimap of rectangular $(a,b,c)$-matrix multiplication:
\begin{align*}
	*:&\M_{a\times b}(k)\times \M_{b\times c}(k)\bmto \M_{a\times c}(k) &  [u_{ij}]*[v_{ij}] & = \left[\Sigma_k u_{ik}v_{kj}\right].
\end{align*}
Rectangular matrices in general have no identity, inverses, conjugation, idempotents, nilpotents, or polynomial identities.
So however ring-like our claims appear, the approach must be different.

\subsection{Main results}
Throughout $k$ is a commutative associative unital ring and all $k$-(bi)modules are unital.
A $k$-bimap $*:U_*\times V_*\bmto W_*$ consists of $k$-bimodules $U_*$, $V_*$, and $W_*$ with the two-sided
distributive law: $(u+u')*v=u*v+u'*v$ and $u*(v+v')=u*v+u*v'$; and the osmosis of scalars $s\in k$: 
$(su)*v=s(u*v)=u*(sv)$.

A {\em homotopism} $\phi$ between bimaps $*:U_*\times V_*\bmto W_*$ and $\bullet:U_{\bullet}\times V_{\bullet}\bmto W_{\bullet}$
is a triple $\phi=(R_{\phi}^U,R_{\phi}^V;R_{\phi}^W)\in \hom(U_*,U_{\bullet})\times \hom(V_*,V_{\bullet})\times \hom(W_*,W_{\bullet})$ where
\begin{align}\label{def:homotopism}
	u\phi * v\phi & = uR_{\phi}^U* vR_{\phi}^V = (u*v)R_{\phi}^W = (u*v)\phi.
\end{align}
Homotopisms form a category with the expected notions of mono-, epi-, iso-, and auto-topisms; and appropriate versions of
Noether's isomorphism theorems are satisfied.
Homotopisms were first used by Albert \cite{Albert:autotopism} for products $*:A\times A\bmto A$ of nonassociative $k$-algebras $A$.
Notice if $A=\bigcup_s A_s$ is a filtered algebra and each $A_s$ is characteristic then every 
automorphism $\phi\in \Aut(A)$ restricts to $R_{\phi}^s \in \Aut(A_s/\partial A_{s})$ and
$(R_{\phi}^s,R_{\phi}^t;R_{\phi}^{s+t})$ is an autotopism of the bimap
$A_s/\partial A_s \times A_t/\partial A_t\bmto A_{s+t}/\partial A_{s+t}$.  

In our notation operators act opposite to scalars (more in Section~\ref{sec:prelims}).
For each bimap $*:U\times V\bmto W$ we have three associative unital algebras called the 
{\em left}, {\em mid}, and {\em right} {\em scalars} (or {\em nuclei} in the nonassociative parlance).  These are:
\begin{align*}
	\LL{*} & = \{ \lambda\in \End(U_k)\times \End(W_k) : \forall u\forall v, (\lambda u)*v=\lambda (u*v) \},\\
	\MM{*} & = \{ \mu\in \End({_k U})\times \End(V_k) : \forall u\forall v, (u\mu)*v=u*(\mu v)\},\textnormal{ and}\\
	\RR{*} & = \{ \rho\in \End({_k V})\times \End({_k W}) : \forall u\forall v, u*(v\rho)= (u*v)\rho \}.
\end{align*}
Let $\LMR_*=\LL{*}\oplus \MM{*}\oplus \RR{*}$ and $Z(\LMR_*)$ be the center.
Historically $\MM{*}$ appears as the ring of adjoints; cf. \citelist{\cite{BFFM}\cite{BW:isom}}.  
Last, the {\em centroid} $\CC{*}$ replaces $k$; cf. \cite{Myasnikov}. 
\begin{align*}
	\CC{*} & = \{ \sigma\in \End(U_{\mathbb{Z}})\times \End(V_{\mathbb{Z}})\times \End(W_{\mathbb{Z}}) : 
			 \forall u\forall v, (u\sigma)*v= (u*v)\sigma=u*(v\sigma)\}.
\end{align*}
When we wish to focus on homotopisms that are linear with respect to
one or more of the rings above we indicate this with a subscript, e.g. $\Aut_{\CC{*}}(*)$ denotes
the $\CC{*}$-linear autotopisms of a bimap $*$.  Notice each of these rings is 
constrained by linear equations and so they can be efficiently computed for specific examples. 
There are highly tuned algorithms for that task described in \cite{BW:slope}.
\medskip

For convenience let us assume $*:U\times V\bmto W$ is {\em fully nondegenerate} in that $u*V=0$ implies $u=0$,
$U*v=0$ implies $v=0$, and $W=U*V$.  Later in Section~\ref{sec:degenerate} we handle the general case.
Our main effort is to introduce structure to autotopism groups that depend on the far better understood
properties of the rings $\LMR_*$ and $\CC{*}$. We prove:
\begin{thm}\label{thm:auto}
We have the following exact sequences of groups:
\begin{align}
\tag{A}
	1\to &  \LL{*}^{\times}  \to \Aut(*)  \to \Aut(V_*),\\
\tag{B}
	1\to &  \MM{*}^{\times}  \to \Aut(*) \to \Aut(W_*),\\
\tag{C}
	1\to & \RR{*}^{\times}  \to \Aut(*)   \to \Aut(U_*),\\
\tag{D}
	1\to & \Aut_{\CC{*}}(*)\to \Aut(*) \to \Out(\CC{*}),\textnormal{ and}\\
\tag{E}
1\to &  Z(\LMR_*)^{\times} \rightarrow \LMR_*^{\times}\times \Aut_{\LMR}(*) \rightarrow \Aut_{\CC{*}}(*) 
	\rightarrow \Out_{\CC{*}}(\LMR_*).
\end{align}
If $e^2=e\in \LMR_*$ such that $\LMR_*=(\LMR_*) e(\LMR_*)$ then 
\begin{align}
\tag{F}
	\Aut_{\LMR}(*:U\times V\bmto W) & \cong \Aut(eUe\times eVe\bmto eWe).
\end{align}
\end{thm}

In Section~\ref{sec:archetype} we apply these technical decompositions of Theorem~\ref{thm:auto} to a wide range of algebraic
systems, particularly nilpotent groups and rings.  One example
is sufficiently elementary for an introduction.
Recall the Skolem-Noether theorem asserts that $k$-linear ring automorphisms of $\M_a(k)$ are inner, i.e.
$\GL_a(k)/k^{\times}$; cf. \cite{Rowen:II}*{p. 460}.   A scholium to Theorem~\ref{thm:auto}(E) is the observation that
the image of $\LMR^{\times}_*$ in $\Aut(*)$, 
as given by equations (A), (B), and (C), induces inner automorphisms acting on $\LMR_*$.
Though bimaps have no notion of conjugation, using $\LMR_*$ we make
a meaningful sense of ``inner autotopisms'' as follows:
\begin{align}\label{def:inner}
	\Inn(*) & = \LMR_*^{\times}/\CC{*}^{\times}.
\end{align}
Using Theorem~\ref{thm:auto} and resolving the necessary computations we prove:

\begin{thm}[Generalized Skolem-Noether]\label{thm:Skolem-Noether}
The $k$-linear autotopisms of  matrix multiplication 
$*:\M_{a\times b}(k)\times \M_{b\times c}(k)\bmto \M_{a\times c}(k)$ over a field $k$ are
\begin{align*}
	\Aut_k(*)  & = \Inn(*) \cong \frac{\GL_a(k)\times \GL_b(k)\times \GL_c(k)}{\{ (sI_a,sI_b,sI_c) : s\in k^{\times}\}}.
\end{align*}
\end{thm}

This thinking makes natural proofs of the following sort. Fix a field $k$, positive integers $a,b,c$, and a 
multiplicatively closed  nonemepty subset $S$ of $k$.  Define
\begin{align*}
	\mathcal{T}_{a,b,c}(k;S) & = \left\{\begin{bmatrix} sI_a & A & C \\ 0 & sI_b & B \\ 0 & 0 & sI_c \end{bmatrix}:s\in S, 
		\begin{array}{c} A\in \M_{a\times b}(k),\\ B\in \M_{b\times c}(k),\\ C\in \M_{a\times c}(k)\end{array}
	\right\}.
\end{align*}
Notice $\mathcal{T}_{a,b,c}(k;k)=k\oplus \mathcal{N}$ is a local unital associative $k$-algebra with nilpotent radical 
$\mathcal{N}=\mathcal{T}_{a,b,c}(k;\{0\})$.  This is as close to nilpotent as a unital algebra can be.  Changing
$S$ we find $\mathcal{T}_{a,b,c}(k;\{1\})$ is a nilpotent group under multiplication.  Also
$\mathcal{T}_{a,b,c}(k;\{0\})$ under commutation is a nilpotent Lie $k$-algebra. 
We prove:

\begin{coro}\label{coro:tri} The associative ring automorphisms of  $\mathcal{T}_{a,b,c}(k;k)$ are the groups
\begingroup
\setlength{\thinmuskip}{0mu}
$$	k^{(ab+bc)(ac)}{\rtimes}\left(
			\frac{\GL_a(k){\times}\GL_b(k){\times}\GL_c(k)}{\langle (s,s,s) : s\in k^{\times}\rangle}			
			\right){\rtimes}\Gal(k).$$
\endgroup
The Lie ring automorphisms of $\mathcal{T}_{a,b,c}(k;\{0\})$ are the groups
\begin{align*}
\begin{array}{lc}
		k^{2b}\rtimes \left({\rm Sp}_{2b}(k)\times k^{\times}\right)\rtimes \Gal(k), &  a=c=1;\\
		k^{(ab+bc)(ac)}\rtimes (\GL_b(k)\times k^{\times})\rtimes \Gal(k), & 1<a,c;\\
		k^{(ab+bc)(ac)}\rtimes \left(k^{ab^2c}\rtimes (\GL_b(k)\times k^{\times})\right)\rtimes \Gal(k), & \textnormal{else}.\\
	\end{array}
\end{align*}
Finally for fields $k=2k$, the group automorphisms of $\mathcal{T}_{a,b,c}(k;\{1\})$ agree with the Lie ring automorphisms
of $\mathcal{T}_{a,b,c}(k;\{0\})$.
\end{coro}

\subsection{A comparison with semisimple contexts}
  Corollary~\ref{coro:tri} is mostly intended as a small demonstration 
of Theorem~\ref{thm:auto}.  The case for $a=c=1$ is well-known as it coincides
with Heisenberg groups and algebras.  However, the general case might not have
been probed before.  One variation studied in detail in \citelist{\cite{AAB}\cite{C}} offers a relevant point of comparison and
shows why nilpotence can be such a difficult starting point.

Consider the rings $\mathcal{B}=\mathcal{B}_{a,b,c}(k)=(\M_a(k)\oplus\M_b(k)\oplus \M_c(k))\oplus \mathcal{N}$ 
of full $(a,b,c)$-block triangular matrices.  
To explain $\Aut(\mathcal{B})$ first apply
Wedderburn's principal theorem so that the decompostion
of $\mathcal{B}$ into $S=\M_a(k)\oplus \M_b(k)\oplus\M_c(k)$
and $\mathcal{N}$ is unique upto conjugation.  By Krull-Schmidt we
can either permute the simple factors, or we induce automorphisms
of the simple rings -- those we know are inner or induced by field automorphisms
because of the usual Skolem-Noether.
Permutations of the simple factors do not lift to $\Aut(\mathcal{B})$ as
that would involves interchanging left and right ideals inside $\mathcal{N}$.
Next, an automorphism that centralizes $S$ will act on $\mathcal{N}$ as an 
$S$-bilinear endomorphism.  By Schur's lemma the action is as scalars 
on each simple factor of the semisimple $S$-bimodule 
$J=\M_{a\times b}(k)\oplus \M_{b\times c}(k)\oplus \M_{a\times c}(k)$.
As a result the only actions are inner automorphisms of $\mathcal{B}$, and coordinatewise field automorphisms. 
Many more general treatments are known \citelist{\cite{AAB}\cite{C}}. 

Now let us compare with our claims in Corollary~\ref{coro:tri}.
Because $\mathcal{T}_{a,b,c}(k;S)$ is local (resp. nilpotent), 
{\em none} of the above cited ingredients (e.g. theorems of
Krull-Schmidt, Wedderburn, Skolem-Noether, Schur) offer any information
to the problem of constructing $\Aut(\mathcal{T}_{a,b,c}(k;S))$.  In fact,
the values of $a,b,c$ and the matrix structure so apparent in the definition
is in no immediate way a parameter recognized by automorphisms.  The groups
that apriori could have been represented as automorphisms might have been
as large as $\GL_n(k)$ where $n=ab+bc$ is the dimension of $\mathcal{N}/\mathcal{N}^2$.
Indeed, those familiar with the constructions of the automorphism group
for Heisenberg groups and algebras -- the case where $a=c=1$ -- know 
that the proof hinges on a completely different and delicate arrangement where
$\mathcal{T}_{1,b,1}(k;S)$ affords a nondegenerate alternating form.
Generalizing to arbitrary $(a,b,c)$ that approach encounters problems of
wild representation type, and quickly becomes unfeasible.  So under scrutiny, 
$\mathcal{T}_{a,b,c}(k;S)$ is not as mundane as might be predicted.
\medskip

In broad strokes the theme of this article is that we can recover 
missing nice structure such as semisimplicity even if on the surface a product appears to be nilpotent.  Our title reflects this philosophy:
Skolem-Noether type theorems can be applied to nilpotent products.

\subsection{Outline of the paper}

We start in Section~\ref{sec:prelims} giving basic definitions and notation.  
In Section~\ref{sec:rings} we inspect the associated rings $\LMR_*=\LL{*}\oplus \MM{*}\oplus \RR{*}$ of a bimap $*$ and 
describe their universal properties and relationship to tensor and versor products (Theorem~\ref{thm:universal}).  
In Sections \ref{sec:aut} we prove Theorem~\ref{thm:auto} parts (A)--(E), but first in Section~\ref{sec:der} we prove an analogous
but easier case for the Lie algebra of derivations of a bimap. 
Then in Section~\ref{sec:Morita} we introduce a theorem akin to Morita 
condensation and prove Theorem~\ref{thm:auto}(F).  
In Section~\ref{sec:universal} we begin 
to transfer the general structure theorems to specific bimaps including tensor and versor products as well as 
proving our Skolem-Neother Theorem~\ref{thm:Skolem-Noether}.
In Section~\ref{sec:sym-deg} we generalize the claims
to symmetric, alternating, and weakly-Hermitian bimaps as well as degenerate bimaps.  
Finally in Section~\ref{sec:archetype}
we explain the implications to isomorphism problems of nilpotent groups and algebras (Figure~\ref{fig:tensor-versor-triangle})
and prove Corollary~\ref{coro:tri}.

%
\section{Preliminaries}\label{sec:prelims}
Throughout we assume $k$ is a commutative unital ring and that $*:U\times V\bmto W$ is a $k$-bimap 
of $k$-bimodules $U$, $V$, and $W$.  We act opposite to scalars, so given $\phi\in \End(_k U)$ we
evaluate $u\in U$ as $u\phi=uR_{\phi}^U$.  Likewise write $\phi u=uL_{\phi}^U$ for $\phi\in \End(U_k)$.
General preliminaries are found here \citelist{\cite{Rowen:I}\cite{Rowen:II}}.

We mentioned that we will initially assume all bimaps are fully nondegenerate.
When this fails we can measure the defect through on or more of the following
{\em radicals}
\begin{align*}
	U^{\bot} & = \{ v\in V: U*v=0\}, \\
	V^{\top} & = \{ u\in U: u*V=0\},\\
	W^{+} & = W/(U*V)=W/\langle u*v : u\in U, v\in V\rangle.
\end{align*}
When all three are trivial we say that $*$ is fully nondegenerate.

A {\em swap} $\tilde{*}$ of a bimap $*:U_*\times V_*\bmto W_*$ is defined by $U_{\tilde{*}}=V_*$, $V_{\tilde{*}}=U_*$, 
$W_{\tilde{*}}=W_*$ and with product $v~\tilde{*}~u = u*v$.  The swap sends homotopisms
$\phi=(R^U_{\phi},R_{\phi}^V; R_\phi^W)$ to $\tilde{\phi} = (R_{\phi}^V,R_{\phi}^U; R_{\phi}^W)$.  

\begin{remark}
In earlier treatments 
\citelist{\cite{Knuth}*{Section 4}\cite{Wilson:division}*{p. 3994}} such deformations where discussed as a generalized ``transpose''.  
This is a contra-variant functor on the adjoint category of bimaps,
but it is co-variant on the homotopism category. A co-variant ``transpose'' fights common practice so we use ``swap'' instead.
\end{remark}

Call a bimap {\em weakly Hermitian} if there is an isotopism $\tau:*\to \tilde{*}$ such that $\tau \tilde{\tau}=1_*$ and $\tilde{\tau} \tau=1_{\tilde{*}}$;
in particular, $U_*\cong V_*$ as $k$-modules.  If $U_*=V_*$ and $R_{\tau}^U=R_{\tau}^V=1$ then we say $*$ is {\em Hermitian}.  
Define the {\em pseudo-isometry} group as
\begin{align}\label{def:pseudo}
	\Psi\Isom(*,\tau) & = \left\{ \phi\in \Aut(*) : \phi \tau = \tau \tilde{\phi}\right\}.
\end{align}
In general $\tau$ can influence the isomorphism type of the group $\Psi\Isom(*,\tau)$, even in the case of classical Hermitian
forms; see \cite{BHRD}.
\medskip

For an associative ring $A$, let $A^\times$ denote the group of units, and $A^{-}$ the Lie ring with product $[x,y]=xy-yx$.
$\End(_k U)$ is the endomorphism ring of ${_k U}$ and
\begin{align*}
	\GL(_k U) & = \End(_k U)^{\times} & 
	\gl(_k U) & = \End(_k U)^{-}.
\end{align*}
$R_{\alpha}^U\mapsto L^U_{\alpha}$ is the anti-isomorphism $\End(_k U)\to\End(U_k)$.  Write $A^{\circ}$ for the $op$-ring of $A$.

%
\section{The universal rings for bimaps}\label{sec:rings}

We describe the rings that arise naturally along side distributive products $*:U\times V\bmto W$.
  We begin by describing the rings that act as left, mid, and right scalars, followed by
the centroid.  We give these rings a universal description (Theorem~\ref{thm:universal}).

\subsection{Scalar rings}
There are at least four reasonable notions of scalar actions on a bimap; we begin with three. 
Fix operators $\LL{}\to \End(U_{\mathbb{Z}})\times\End(W_{\mathbb{Z}})$, 
$\MM{}\to \End({_{\mathbb{Z}} U})\times \End(V_{\mathbb{Z}})$, and $\RR{}\to \End({_{\mathbb{Z}} V})\times \End({_{\mathbb{Z}} W})$.
Call $*$ {\em left $\LL{}$-linear}, {\em mid $\MM{}$-linear}, and {\em right $\RR{}$-linear} if each of the respective properties holds.
\begin{align*}
	(\forall & \lambda\in \LL{}) & (\lambda u)*v & = \lambda (u*v),\\
	(\forall & \mu\in \MM{}) & (u\mu)*v & =  u*(\mu v),\\
	(\forall & \rho\in \RR{}) & u*(v\rho) & =  (u*v)\rho.	
\end{align*}
An {\em $\LMR$-bimap} indicates the simultaneous left, mid, and right
linearities.  Evidently every bimap $*$ is automatically an $\LMR_*$-bimap for the ring
$\LMR_*=\LL{*}\oplus \MM{*}\oplus \RR{*}$ provided in the introduction.
Those familiar with nonassociative ring theory should compare $\LL{*}$, $\MM{*}$, and $\RR{*}$ 
to the left, middle, and right nuclei of a ring \cite{Schafer}*{Chapter I}.

\begin{remark}\label{rem:notation}
We will mostly conform to the conventions just demonstrated, and we will also 
write $\lambda=(L_{\lambda}^U,L^U_{\lambda})$, or $\mu=(R_{\mu}^U,L_{\mu}^V)$, etc. to indicate 
explicitly what action is under consideration. The need to be particular in notation is that most of
our arguments are systematic verification of carefully arranged definitions.  Such automatic 
proofs become unreasonable to reconstruct if the understanding of left/right actions becomes too muddled.
\end{remark}

\begin{ex}
\begin{enumerate}[(i)]
\item If $*:A\times A\bmto A$ is the product of an associative unital ring $A$ then $\LL{*}\cong A^{\circ}$ and
$\MM{*}\cong \RR{*}\cong A$.
\item  If $*:A\times M\bmto M$ is the product of a faithful associative left $A$-module $M$ 
then $\RR{*}=\End({_A M})$.
\item  If $*:M\times M\bmto k$ is a nondegenerate $k$-bilinear form 
then $\MM{*}$ is the usual ring of adjoints and $\LL{*}$ and $\RR{*}$ are copies of $k$.
\end{enumerate}
\end{ex}
\begin{proof}
(i) Fix $\lambda\in \LL{*}$.  Set $a=\lambda 1=1L_{\lambda}^U$ (recall we use $U$ to indicate position, even
though in this case $U=V=W=A$.)  Now for all $b\in A$, $bL_{\lambda}^W=(1*b)L_{\lambda}^W=(1L_{\lambda}^U)*b=a*b$.
So $L_{\lambda}^W=L_a$.  Likewise, $L_{\lambda}^U=L_a$.  So $\LL{*}\cong A^{\circ}$.  The rest follows.  Both (ii) and (iii)
can be seen directly from the definitions.
\end{proof}

An important example throughout is that the left, mid, and right scalar rings
of $(a,b,c)$-matrix multiplication are 
 $\M_a(k)$, $\M_b(k)$, and $\M_c(k)$ respectively.  While it is evident that these rings act appropriately, it remains to show equality.

Let $U\otimes_{\MM{}} V$ be the usual Whitney tensor product, and put 
\begin{align*}
	U\;{_{\LL{}}\lversor}\; W & = \hom(_{\LL{}} U,_{\LL{}} W) & 
	W\rversor_{\RR{}} V & = \hom(V_{\RR{}},W_{\RR{}})
\end{align*}
We call $\lversor$ and
$\rversor$ left, resp. right, \emph{versor products}.  Just as with tensor products we  have associated bimaps 
${_{\LL{}}\lversor}:U\times U\;{_{\LL{}}\lversor}\; W\bmto W$ and $\rversor_{\RR{}}:W\rversor_{\RR{}} V\times V\bmto W$ of evaluation.
We  suppress scalars when context permits.
See also \cite{Wilson:division}*{Section 2.3}.

\begin{prop}
Every $\LMR$-bimap $*:U\times V\bmto W$ determines unique additive homomorphisms
\begin{align*}
	\cev{*}:& U\to W\rversor_{\RR{}} V & u & \mapsto (v\mapsto u*v)\\
	\hat{*}:&U\otimes_{\MM{}} V\to W &  u\otimes v & \mapsto u*v\\
	\vec{*}:& V\to U{_{\LL{}}\lversor} W &  v &\mapsto (u\mapsto u*v).
\end{align*}
Furthermore, $\cev{*}$ is $(\LL{},\MM{}^{\circ})$-linear, $\hat{*}$ is $(\LL{},\RR{})$-linear, and $\vec{*}$ is $(\MM{}^{\circ},\RR{})$-linear.
\end{prop}

A sloppy but convenient practice will be to also write $\cev{*}$ for the canonical homotopism 
$(\cev{*},1_V;1_W)\in \hom(*,\rversor_{\RR{}})$. Likewise overload $\hat{*}$ with 
$(1_U,1_V;\hat{*})\in \hom(\otimes_{\MM{}},*)$ and 
$\vec{*}$ with $(1_U,\vec{*};1_W)\in \hom(*,{_{\LL{}}\lversor})$.  

\begin{thm}[Universality of scalar rings]\label{thm:universal}
Every bimap $*:U\times V\bmto W$ is an $\LMR_*$-bimap.  Furthermore, if $*$ is also an $\LMR$-bimap, then the image of
the representation $\LL{}\to \End(U_{\mathbb{Z}})\times\End(W_{\mathbb{Z}})$ lies in $\LL{*}$ and
$\vec{*}:V\to U\;{_{\LL{}}\lversor}\;W$ factors through $U\;{_{\LL{*}}\lversor}\; W$.  Similarly,
$\MM{}\to \MM{*}$ and $\RR{}\to \RR{*}$; and $\hat{*}$ and $\cev{*}$ factor through $U\otimes_{\MM{*}} V$ and
$W\rversor_{\RR{*}} V$ respectively; see Figure~\ref{fig:scalar-diagram}.
\end{thm}

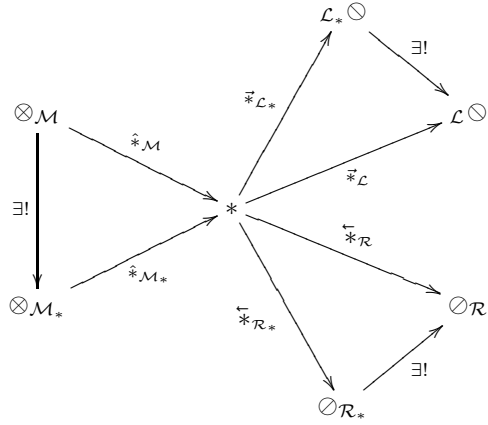
\begin{figure}
\begin{align*}
\xymatrix{
&	& & {_{\LL{*}}\lversor}\ar[dr]^{\exists !} & \\
\otimes_{\MM{}} \ar[rrd]^{\hat{*}_{\MM{}}} \ar[dd]_{\exists !} & & &  & {_{\LL{}}\lversor}\\
 &    & {*} \ar[ruu]^{\vec{*}_{\LL{*}}}\ar[rdd]_{\cev{*}_{\RR{*}}}\ar[urr]_{\vec{*}_{\LL{}}}\ar[drr]^{\cev{*}_{\RR{}}}\\
\otimes_{\MM{*}}\ar[rru]_{\hat{*}_{\MM{*}}} & &  &  & \rversor_{\RR{}}\\
& 	& & \rversor_{\RR{*}} \ar[ur]_{\exists !} & 
}
\end{align*}
\caption{Commutative diagram in the homotopism category concerning the universality of scalars.}\label{fig:scalar-diagram}
\end{figure}

\subsection{Centroids, bimodules, and 3-pile-shuffles}
We introduce a method of proof so common to bimaps that we give it a name: {\em 3-pile-shuffling}.  
We introduce this by demonstrating
with the the fourth form of linearity in a bimap: {\em bi-}linearity.  Here we have a commutative ring $k$ represented
in $\End(U)\times \End(V)\times \End(W)$ such that 
\begin{align*}
	(\forall & s\in k, \forall u\in U, \forall v\in V) & (su)*v & = s(u*v) = u*(sv).
\end{align*}
We should like to observe that every $k$-bimap is also a $\CC{*}$-bimap and that there is canonical ring
homomorphism $k\to \CC{*}$ compatible with these two interpretations.  Of course this is true, but a technical point
stands in the way which is that we have defined $k$-bilinearity as acting on the left of $U$, $V$, and $W$, and our
convention to ``act opposite scalars'' made our default definition of $\CC{*}$ represent a right module action.
The following shuffling of the letters between the three groups $U$, $V$, and $W$ resolves the problem.
\begin{align*}
	(\forall & s,t\in k,\forall u\in U,\forall v\in V) &
	((st)u)*v & = s(u*(tv)) = ((ts)u)*v
\end{align*}
which implies that $(st-ts)U\in V^{\bot}$.  Indeed, as long as $*$ is fully nondegenerate then the commutator $[k,k]$ is in the annihilator of
each of the groups $U$, $V$, and $W$, and so we harmlessly insist that $k$ is commutative; compare \cite{Myasnikov}. So in bilinearity
the distinction between left and right is inconsequential.
\smallskip

Though imprecise, we will regard the general process of moving scalars between the various groups $U$, $V$ and $W$
of a bimap $*:U\times V\bmto W$ as ``3-pile-shuffles''.
\smallskip

Using 3-pile-shuffles one further concludes that the $\CC{*}$-action commutes with that of $\LL{*}$, $\MM{*}$, and $\RR{*}$.
Further still, ${_{\LL{*}} U}_{\MM{*}}$ is essentially a bimodule in the sense that 
\begin{align*}
	(\lambda u)\mu - \lambda (u\mu) & \in V^{\bot}.
\end{align*}
Similar claims hold for  ${_{\MM{*}} V}_{\RR{*}}$ and ${_{\LL{*}} W}_{\RR{*}}$.   In particular, if $*$ is fully nondegenerate then
everyone of these is an honest bimodule.  We have proved:

\begin{prop}
For an $\LMR$-bimap $*:U\times V\bmto W$, for all $u\in U$, all $v\in V$, and all 
 $(\lambda,\mu,\rho)\in \LMR$,
\begin{align*}
	(\lambda u)\mu * v & = \lambda (u\mu)* v &
	(\lambda u* v)\rho & =  \lambda(u* v\rho) &
	u* (\mu v)\rho & = u* \mu (v\rho).
\end{align*}
If $*$ is fully nondegenerate then $U$ is an $(\LL{},\MM{})$-bimodule, 
$V$ is an $(\MM{},\RR{})$-bimodule, and $W$ is an $(\LL{},\RR{})$-bimodule.
\end{prop}

%
\section{The structure of derivations of bimaps}\label{sec:der}
We prove structure theorems for the Lie algebra of derivations of a bimap.  These are slightly simpler 
demonstrations of the structure found in autotopism groups.
Throughout we fix a fully-nondegenerate $k$-bimap $*:U\times V\bmto W$. Define the following Lie algebra of
{\em derivations}
\begin{align*}
	\Der_k(*) & = \{ \delta\in \gl(_k U)\oplus \gl(_k V)\oplus \gl(_k W) : \forall u\forall v, (u\delta)*v+u*(v\delta)= u*^{\delta} v\}.
\end{align*}
Derivations in this general form have been developed for use in studying algebras.  For instance they are essential
in describing Cartan-Jacbson principal of triality (\cite{Schafer}*{Theorem 3.31}) and are also used in the study of
solvable and nilpotent Lie algebras; see \cite{LL:gen-der} and accompanying citations.  The adaptation to bimaps appears in
\cite{Wilson:alpha}*{Section 4.2}, but quite possibly it arose long before that work.

It will be helpful in this section to have understood the notation in Remark~\ref{rem:notation}.
\medskip

So far $\Der_k(*)$ has no direct awareness of the universally described scalar rings $\LL{*}$, $\MM{*}$, or $\RR{*}$
we have seen above.  Indeed, derivation rings are not part of any known functorial connection and can vary quite wildly with small changes in
a bimap.  So the inclusion of scalars is, at least on the surface, an artificial imposition as follows:
\begin{align*}
	\Der_{\LMR}(*) 
	& = \{\delta \in \gl({_{\LL{*}} U}_{\MM{*}})\times \gl({_{\MM{*}} V}_{\RR{*}}) \times \gl({_{\LL{*}} W}_{\RR{*}}) :\\
		& \qquad \forall u\forall v,		 (u\delta)*v+u*(v\delta) = u*^{\delta} v\}.
\end{align*}
Our principle concern in this section is to demonstrate that we can largely predict the structure $\Der_k(*)$  from $\Der_{\LMR}(*)$ and
an understanding of the associative unital rings $\LL{*}$, $\MM{*}$, and $\RR{*}$.  
\medskip

As above, for  (non-)associative rings $A$ let $\Der_k A$ denote the derivations of the product of $A$,
$\ad A=\ad A^{-}$ the ``inner'' derivations, and $\Der_k A/\ad A$ we consider the ``outer'' derivations.
We prove:

\begin{thm}\label{thm:der-2}
Fix a fully nondegenerate $k$-bimap $*:U\times V\bmto W$.  
There is an exact sequence of Lie $k$-algebras
\begin{align*}
	0\to Z(\LMR_*)^- \to \LMR_{*}^{-}\oplus \Der_{\LMR}(*)\longrightarrow \Der_k(*)
		\longrightarrow \frac{\Der_k\LMR_*}{\ad \LMR_*}.
\end{align*}
\end{thm}

First we pause to recognize that every $k$-bimap is also a $\mathbb{Z}$-bimap and a $\CC{*}$-bimap, so $k$ is
not canonical.  Yet different choices of $k$ can be reconciled as follows.

\begin{prop}
There is an exact sequence of Lie $k$-algebras
\begin{align*}
	0\to \Der_{\CC{*}}(*)\to \Der_k(*)\to \Der_k(\CC{*}).
\end{align*}
\end{prop}
\begin{proof}
Let $\sigma\in \CC{*}$ and $\delta\in \Der_k(*)$.  Define
\begin{align*}
	[\sigma,\delta] & = ([R_{\sigma}^U,R_\delta^U], [R_\sigma^V,R_\delta^V]; [R_\sigma^W, R_\delta^W]).
\end{align*}
Observe that $[\sigma,\delta]\in \CC{*}$ because
\begin{align*}
	u[R_{\sigma}^U,R_\delta^U] * v 
		& = (u\sigma*v)\delta-(u\sigma)*(v\delta)-(u*v\sigma)\delta+u*(v\sigma\delta) = u*(v[R_{\sigma}^W,R_{\delta}^V]),\\
	u[R_{\sigma}^U,R_\delta^U] * v 	
		& = (u\sigma*v)\delta	-(u\sigma)*(v\delta) -(u*v)\delta\sigma + u*(v\delta\sigma)
		= (u*v)[R_\sigma^W,R_\delta^W].
\end{align*}
Hence, $[\sigma,\delta]\in \CC{*}$ and $[-,\delta]$ is a derivation of $\CC{*}$ (as an associative ring).
\end{proof}

Now we prove the main claim of this section in stages.

\begin{prop}\label{prop:der-1}
There is an exact sequence of Lie $k$-algebras
\begin{align*}
	0\to \CC{*}^{-} \overset{\chi}{\longrightarrow} \LMR_*^{-}\overset{\Delta}{\longrightarrow} \Der_k(*).
\end{align*}
Furthermore, the restriction of $\Delta$ induces the following short exact sequences.
\begin{align*}
	0 & \to \LL{*}^{-} \overset{\Delta}{\longrightarrow} \Der_k(*)\to \Der_k(*)|^V\to 0\\
	0 & \to \MM{*}^{-} \overset{\Delta}{\longrightarrow} \Der_k(*)\to \Der_k(*)|^W\to 0\\
	0 & \to \RR{*}^{-} \overset{\Delta}{\longrightarrow} \Der_k(*)\to \Der_k(*)|^U\to 0
\end{align*}
In particular the image $\ad(*)$ of $\Delta$ is an ideal of $\Der_k(*)$ and $\Der_k(*)$ 
contains a faithful copy of the abelian Lie $k$-algebra.
$\CC{*}^{-}\oplus \CC{*}^{-}\neq 0$.
\end{prop}
\begin{proof}
Let $\Delta$ map $\LMR_*^{-}$ into $\End(U)\oplus \End(V)\oplus \End(W)$ by
\begin{align*}
	\left((L_\lambda^U,L_\lambda^W),(R_\mu^U,L_\mu^V), (R_\rho^V,R_\rho^W)\right)
		& \mapsto (R_\mu^U-L_\lambda^U,R_\rho^V-L_\mu^V; R_\rho^W-L_\lambda^W).
\end{align*}
For all $u\in U$ and all $v\in V$, 
\begin{align*}
	u(R_\mu^U-L_\lambda^U)*v+u*v(R_\rho^V-L_\mu^V)
		& = u*v\rho-\lambda u*v
		 = (u*v)(R_\rho^W-L_\lambda^W).
\end{align*}
So the image of $\Delta$ lies in $\Der_k(*)$.   Also $\Delta$ is $k$-linear.   
Since $*$ is fully nondegenerate the three-pile shuffle applies to say that 
$U$, $V$, and $W$ are $(\LL{*},\MM{*})$-, $(\MM{*},\RR{*})$-, and $(\LL{*},\RR{*})$-bimodules
respectively.  Thus, for example, $[R_\mu^U,L_{\lambda'}^U]=0$. Hence,
\begin{align*}
	[(\lambda,\mu,\rho),(\lambda',\mu',\rho')]\Delta 
		& = \left(\left[R_\mu^U,R_{\mu'}^U\right]-\left[R_\mu^U,L_{\lambda'}^U\right]
				-\left[L_\lambda^U,R_{\mu'}^U\right]+\left[L_\lambda^U,L_{\lambda'}^U\right],\right.\\
		& \qquad	\left[R_\rho^V,R_{\rho'}^V\right]-\left[R_\rho^V,L_{\mu'}^V\right]
					-\left[R_{\rho'}^V,L_{\mu}^U\right]+\left[L_\mu^V,L_{\mu'}^V\right],\\
		& \qquad	\left.\left[R_\rho^W,R_{\rho'}^W\right]-\left[R_\rho^W,L_{\lambda'}^W\right]
						-\left[R_{\rho'}^W,L_\lambda^W\right]+\left[L_\lambda^W,L_{\lambda'}^W\right]\right)\\
		& = [(\lambda,\mu,\rho)\Delta,(\lambda',\mu',\rho')\Delta]
\end{align*} 
Therefore $\Delta$ is a Lie homomorphism.

Next we define an associative ring homomorphism 
 $\chi$ from $\CC{*}$ to
$\LMR_*$ which consequently is also a Lie homomorphism $\CC{*}^{-}\to \LMR_*^{-}$.
As $\CC{*}$ is commutative, the anti-isomorphism $R_\sigma^U\mapsto L_{\sigma}^U$
restricted to $\sigma\in \CC{*}$ is an isomorphism (the identity on $\CC{*}|^U$). Define $\chi$:
\begin{align*}
	(R_\sigma^U,R_\sigma^V; R_\sigma^W)\chi & = 
		((L_\sigma^U,L_\sigma^W),(R_\sigma^U,L_\sigma^V),(R_\sigma^V,R_\sigma^W)).
\end{align*}
Notice $\chi$ is a monomorphism of associative (also Lie) rings.
Also, $\chi\Delta=0$.  Indeed, if $(\lambda,\mu,\rho)\Delta=0$ then $L_\lambda^U=R_\mu^U$, 
$R_\rho^V=L_\mu^V$, and $L_\lambda^W=R_\rho^W$.  Hence,
\begin{align*}
	uR_{\mu}^U * v & = \lambda u *v = (u*v) R_\rho^W  = u*(vR_\rho^V).
\end{align*}
That is, $(R_\mu^U,R_\rho^V;R_\rho^W)\in \CC{*}$.  Furthermore,
$(R_\mu^U,R_\rho^V;R_\rho^W)\chi=(\lambda,\mu,\rho)$.  So the sequence $\chi,\Delta$ is exact.

Next observe that restriction of $\Delta$ to $\LL{*}^{-}$ is monic and has image which is trivial on $V$.
Indeed, if $\delta\in \Der(*)$ and $R_\delta^V=0$ then $uR_\delta^U*v=(u*v)R_\delta^W$, 
so that $(L_\delta^U,L_\delta^W)\in \LL{*}$.  Hence, $\Der_k(*)\to \Der_k(*)|^V$ has kernel $\LL{*}^{-}\Delta$.
Likewise with $\MM{*}$ and $\RR{*}$.

Lastly, $*$ is fully nondegenerate and so $\CC{*}$ embeds as a subring of $\LL{*}$, $\MM{*}$, and $\RR{*}$.  So 
$\CC{*}^{-}\oplus \CC{*}^{-}\oplus \CC{*}^{-}\subseteq \LL{*}^{-}\oplus \MM{*}^{-}\oplus \RR{*}^{-}$.
As the kernel of $\Delta$ is a single copy of $\CC{*}^{-}$, $\Der(*)$ contains a faithful copy of
$\CC{*}^{-}\oplus \CC{*}^{-}$ (a nontrivial set but with trivial commutator).
\end{proof}

Define
\begin{align}\label{def:ad}
	\ad(*) = (\LMR_*^-)\Delta \cong \LMR_*^{-}/\CC{*}^{-}.
\end{align} 
By Proposition~\ref{prop:der-1} $\ad(*)$ is an ideal of $\Der_k(*)$.  Hence, $\Der_k(*)$ acts 
on $\ad(*)$ making $\LMR_*^-$ into a Lie $\Der_k(*)$-module. This is a
symptom of an even more fortunate event: $\Der_k(*)$ acts as derivations on the {\em associative} structure of $\LMR_*$.
We prove:

\begin{thm}\label{thm:der-der}
There is an exact sequence of Lie $k$-algebras
\begingroup
\setlength{\thinmuskip}{0mu}
\begin{align*}
	0\to \Der_{\LMR_*}(*)\hookrightarrow \Der_k(*) \overset{\Xi}{\longrightarrow} \Der_k(\LMR_*).
\end{align*}
\endgroup
Furthermore $0<\CC{*}^{-}\oplus \CC{*}^{-}$ embeds in $\Der_{\LMR_*}(*)$.
\end{thm}
\begin{proof}
Define $\Xi_\delta$ to map $(\lambda,\mu,\rho)\in \LMR_*$ to
\begin{align*}
	\left(\left(\left[L_\delta^U,L_\lambda^U\right],\left[L_\delta^W,L_\lambda^W\right]\right), 
							\left(\left[R_\mu^U,R_\delta^U\right],\left[L_\delta^V,L_\mu^V\right]\right),
							\left(\left[R_\rho^V,R_\delta^V\right],\left[R_\rho^W,R_\delta^W\right]\right)\right).				
\end{align*}
$\Xi$ is linear in $\delta$ and furthermore,  for all $u\in U$ and all $v\in V$,
\begin{align*}
	\left(\left[L_\delta^U,L_\lambda^U\right]u\right)*v 
		& = (\lambda u*v)\delta-\lambda u*(v\delta)-\lambda ((u*v)\delta)+\lambda u*(v\delta)\\
		& = \left[L_\delta^W,L_\lambda^W\right](u*v). 
\end{align*}
Likewise $\left(u\left[R_\mu^U,R_\delta^U\right]\right)*v=u*\left(\left[L_\delta^V,L_\mu^V\right]v\right)$ and
$u*(v\left[R_\rho^V,R_\delta^V\right])=(u*v)\left[R_\rho^W,R_{\delta}^W\right]$.  
So $(\lambda,\mu,\rho)\Xi_{\delta}\in \LMR_*$.

Next, fix $(\lambda,\mu,\rho), (\lambda',\mu',\rho')\in \LMR_*$ and
$\delta\in \Der_k(*)$.  To abbreviate notation we let $\lambda \delta=(\lambda,0,0)\Xi_\delta$, $\mu\delta=(0,\mu,0)\Xi_\delta$, and
$\rho\delta=(0,0,\rho)\Xi_\delta$.
Using the  associative product in the ring $\LMR_*$ we find
\begin{align*}
	(\lambda\lambda')\delta 
	& = \left( [L_\delta^U, L_\lambda^U L_{\lambda'}^U]-L_\lambda^U L_\delta^U L_{\lambda'}^U
			+L_\lambda^U L_\delta^U L_{\lambda'}^U,\right.\\
	& \qquad	\left.\left[L_\delta^U,L_\lambda^W L_{\lambda'}^W\right]-L_\lambda^W L_\delta^U L_{\lambda'}^W+L_{\lambda}^W L_\delta^W L_{\lambda'}^W\right)\\
	& =(\lambda\delta)\lambda'+\lambda (\lambda'\delta)
\end{align*}
Also, $(\mu\mu')\delta = (\mu\delta)\mu'+\mu(\mu'\delta)$ and $(\rho\rho')\delta=(\rho\delta)\rho'+\rho(\rho'\delta)$.
Thus $\Xi_\delta\in \Der_k(\LMR_*)$.  Indeed, $\Xi$ is a Lie homomorphism.

The kernel of the action of $\Der_k(*)$ on $\LMR_*$ by definition contains those operators in $\Der_k(*)$ that
commute with each of the scalar actions by $\LMR_*$.  Hence it is exactly $\Der_{\LMR_*}(*)$. 
Since $\CC{*}$ commutes with $\LMR_*$, the embedding of $\CC{*}^{-}\oplus \CC{*}^{-}\oplus \CC{*}^{-}\to \LMR_*^{-}$
followed by $\Delta$ (Proposition~\ref{prop:der-1}) maps into $\Der_{\LMR_*}(*)$.
\end{proof}

We use the following convention.  For $X,Y\in \End(_k U)$, $\ad_X Y=[X,Y]$, and for $X,Y\in \End(U_k)$, 
$\ad_X Y=[Y,X]=-[X,Y]$.  We saw in Theorem~\ref{thm:der-der} that $\Xi:\Der_k(*)\to \Der_k(\LMR_*)$ and in 
Proposition~\ref{prop:der-1} that $\Delta:\LMR_*^-\to \Der_k(*)$ whose image we defined as $\ad(*)$.  
Now we consider their composition.  We prove:
\begin{thm}\label{thm:ad}
For $(\lambda,\mu,\rho)\in \LMR_*$, $(\lambda,\mu,\rho)\Delta\Xi =(\ad_{\lambda},\ad_{\mu},\ad_{\rho})$.
In particular, $(\LMR_*^-)\Delta\Xi=\ad(*)\Xi=\ad(\LMR_*)$ and $\ker \Delta\Xi=Z(\LMR_*)$.
\end{thm}
\begin{proof}
Fix $\lambda\in \LL{*}^{-}$, set $\delta_{\lambda}=(\lambda,0,0)\Delta\Xi$.  
Note that $(\lambda,0,0)\Delta=(-L_\lambda^U,0;-L_\lambda^W)\in \Der(*)$.
As $V$ and $W$ are respectively $(\MM{*},\LL{*})$- and $(\LL{*},\RR{*})$-bimodules it follows that
for all $(\lambda',\mu,\rho')\in LMR_*$,
\begingroup
\setlength{\thinmuskip}{0mu}
\begin{align*}
	(\lambda',\mu',\rho')\delta_{\lambda} 
		& = \left(\left(\left[L_{\lambda'}^U,L_\lambda^U\right],\left[L_{\lambda'}^W,L_{\lambda}^W\right]\right),
				\left(\left[R_\mu^U,-L_{\lambda}^U\right],\left[0,L_\mu^V\right]\right), 
				\left(\left[R_\rho^V,0\right],\left[R_\rho^W,-L_{\lambda}^W\right]\right)\right)\\
		& = (\ad_{\lambda}(\lambda'), 0, 0)
		= \ad_{(\lambda,0,0)}(\lambda',\mu',\rho').
\end{align*}
\endgroup
By moving similarly through $\MM{*}$ and $\RR{*}$ we confirm that for $(\lambda,\mu,\rho)\in \LMR_*^{-}$, 
$(\lambda,\mu,\rho)\Delta\Xi =(\ad_{\lambda},\ad_{\mu},\ad_{\rho})$. 
Lastly,
\begin{align*}
	\ker \Delta\Xi & \cong \{ (\lambda,\mu,\rho)\in \LMR_*: (\ad_\lambda,\ad_\mu,\ad_\rho)=0\}=Z(\LMR_*).
\end{align*}
\end{proof}

\begin{proof}[Proof of Theorem~\ref{thm:der-2}]
We begin by extending $\Delta$ to $\hat{\Delta}:\LMR_*^{-}\oplus \Der_{\LMR_*}(*)\to \Der(*)$ by 
$(\lambda,\mu,\rho)\oplus \delta\mapsto (\lambda,\mu,\rho)\Delta-\delta$.  The kernel
consists of $(\lambda,\mu,\rho)\oplus (\lambda,\mu,\rho)\Delta$ where
$(\lambda,\mu,\rho)\Delta\in \Der_{\LMR_*}(*)$.  By Theorem~\ref{thm:der-der}, $\Der_{\LMR_*}(*)\Xi=0$
and so we need that $(\lambda,\mu,\rho)\Delta\Xi=0$.
By Theorem~\ref{thm:ad}, 
\begin{align*}\
	\ker\hat{\Delta}=\{(\zeta,\zeta\Delta) : \zeta\in Z(\LMR_*)\}\cong Z(\LMR_*^-).
\end{align*}
So we have an exact sequence $0\to Z(\LMR_*^-)\to \LMR_*^-\oplus \Der_{\LMR}(*)\overset{\hat{\Delta}}{\to} \Der_k(*)$.

Next by Theorems~\ref{thm:der-der} and \ref{thm:ad}, $\Der_{\LMR}(*)\Xi=0$ and $(\LMR_*^-)\Xi=\ad(\LMR_*)$.
So $\im \hat{\Delta}\Xi=\ad(\LMR_*^-)$.  Since $\ad(\LMR_*)$ is an ideal of $\Der(\LMR_*)$ we can induce
$\bar{\Xi}:\Der_k(*)\to \Der_k(\LMR_*)/\ad(\LMR_*)$.  Now $\ker\bar{\Xi}=\im\hat{\Delta}$ and we have
confirmed the existence of the exact sequence state in Theorem~\ref{thm:der-2}.
\end{proof}

%
\section{The Structure of autotopisms of bimaps}\label{sec:aut}
We now prove Theorem~\ref{thm:auto} (A)--(E).
Once more, we assume familiarity with our notation as discussed in Remark~\ref{rem:notation}.
Throughout let us assume $*:U\times V\bmto W$ is fully nondegenerate.

The approach is analogous to our work with derivations in Section~\ref{sec:der}. 
First, it is sufficient to work with $k=\CC{*}$ because of the following exact sequence:
\begin{align}
\tag{D}
	1\to \Aut_{\CC{*}}(*)\to \Aut_k(*)\to \Aut_k(\CC{*})=\Out_k(\CC{*})
\end{align}
Since $*$ is fully nondegenerate $\CC{*}$ is commutative.  Hence $\Aut_k(\CC{*})=\Out_k(\CC{*})$.
So to change scalars requires an extension by a group of outer automorphisms of $\CC{*}$.
We will keep all claims relative to $k$ for uniformity.  This explains Theorem~\ref{thm:auto}(D).

\begin{thm}\label{thm:scalar-aut}[(Theorem~\ref{thm:auto}(A)--(C)]
For every bimap $*:U\times V\bmto W$, there is an exact sequence
\begin{equation*}
\xymatrix{
1\ar[r] & \CC{*}^\times\ar[r]\ar[r]^{\mathfrak{F}} & \LMR_*^\times \ar[r]^{\mathfrak{G}} & \Aut_k(*)\to 1.
}
\end{equation*}
Furthermore there are short exact sequences characterizing the induced actions by
$\LL{*}$, $\MM{*}$, and $\RR{*}$ respectively, as follows.
\begin{align}
\tag{A}
	1 & \to \LL{*}^\times \to \Aut_k(*) \to \Aut_k(*)|^V\to 1, \\
\tag{B}
	1 & \to \MM{*}^\times \to \Aut_k(*) \to \Aut_k(*)|^W\to 1,\textnormal{ and} \\
\tag{C}
	1 & \to \RR{*}^\times \to \Aut_k(*) \to \Aut_k(*)|^U\to 1.
\end{align}
\end{thm}
\begin{proof}
Set $\sigma \mathfrak{F} = ((L_{\sigma}^U,L^W_{\sigma}), (R^U_{\sigma},L^V_{\sigma}),(R^V_{\sigma}, R^W_{\sigma}))$.
As $\CC{*}$ is commutative, for $X\in \{U,V,W\}$, $R_{\sigma}^X\mapsto L_\sigma^X$ is the identity.
Hence, $uL_{\sigma}^U*v  = uR_\sigma^U*v=(u*v)R_\sigma^W=(u*v)L_{\sigma}^W$.
Indeed, $(L_\sigma^U,L_\sigma^W)\in \LL{*}^\times$ and by similar treatment in $\MM{*}$, $\RR{*}$ we see that 
$\sigma\mathfrak{F}\in \LMR_*^\times$.
That $\mathfrak{F}$ is a homomorphism again follows from the commutativity.  Finally, observe that $\sigma\mathfrak{F}=0$
forces $\sigma=0$.

Next define
\begin{align*}
	(\lambda,\mu,\rho)\mathfrak{G} & =
		(R^U_{\mu}L^U_{\lambda^{-1}},R^V_{\rho}L^V_{\mu^{-1}};R^W_{\rho}L^W_{\lambda^{-1}}).
\end{align*}
A 3-pile shuffle confirms this is a homotopism, and the inverse is $(\lambda,\mu,\rho)^{-1}\mathfrak{G}$.
For the homomorphism property recall that $*$ is fully nondegenerate so each the left and right actions
commute (another 3-pile shuffle).  In particular we observe
\begin{align*}
	(\lambda\lambda',\mu\mu',\rho\rho')\mathfrak{G} 
	& = (R_{\mu}^U L_{\lambda^{-1}}^U  R_{\mu'}^UL_{{\lambda'}^{-1}}^U,
			R^V_{\rho} L^V_{\mu^{-1}}R^V_{\rho'}  L^V_{{\mu'}^{-1}}; 
			R_{\rho}^W L_{\lambda^{-1}}^W   R_{\rho'}^W L^W_{{\lambda'}^{-1}})\\
	& = (\lambda,\rho,\mu)\mathfrak{G}(\lambda',\mu',\rho')\mathfrak{G}.
\end{align*}
So $\mathfrak{G}$ is a homomorphism.

Now if $\sigma\in \CC{*}^\times$ then 
\begin{align*}
	\sigma\mathfrak{F}\mathfrak{G}
		& = ((L_{\sigma}^U,L^W_{\sigma}), (R^U_{\sigma},L^V_{\sigma}),(R^V_{\sigma}, R^W_{\sigma}))\mathfrak{G}
		 = (R^U_{\sigma} L^U_{\sigma^{-1}}, R^V_\sigma L^V_{\sigma^{-1}}; R^W_\sigma L^W_{\sigma^{-1}})=1.
\end{align*}
So we have a chain complex.  Finally, if $(\lambda,\mu,\rho)\mathfrak{G}=1$ then $R^U_{\mu}=L_\lambda^U$,
$R^V_\rho =L^V_\mu$, and $R^W_\rho=L_\lambda^W$.  So
\begin{align*}
	uR_\mu^U*v & = \lambda u*v=(u*v)R_\rho^U = u*vR_\rho^V.
\end{align*}
Thus, $(R_\mu^U,R_\rho^V;R_\rho^W)\in \CC{*}^\times$ and 
$(R_\mu^U,R_\rho^V;R_\rho^W)\mathfrak{F}=(\lambda,\mu,\rho)$.  So the first sequence is exact.  

For the next three sequences it suffices to consider the induced representation $\Aut(*)|^V$
of $\Aut(*)$ acting on $V$.  Evidently
$(\lambda, 1,1)\mathfrak{G}=(L_{\lambda^{-1}}^U,1;L_{\lambda^{-1}}^U)$ lies in the kernel of
this representation.  Next suppose  $\phi\in \Aut(*)$ such that $R_{\phi}=1$.  Then
\begin{align*}
	(uL_{\phi^{-1}}^U) * v & = (u*v\phi)\phi^{-1}=(u*v)L_{\phi^{-1}}^W.
\end{align*}
So $(L_{\phi^{-1}}^U,L_{\phi^{-1}}^W)\in \LL{*}^\times$.  In fact 
$(\phi^{-1},1,1)\mathfrak{G}  =
		(L_{\phi}^U, 1; L_{\phi}^W)=(R_\phi^U,R_\phi^V; R_\phi^W)=\phi$.
The sequence is exact, and the others follow likewise.  
\end{proof}

Similar to our above definition \eqref{def:ad} for $\ad(*)$, we 
ascribe the notation and vocabulary of ``inner'' actions to $\LMR_*^{\times}$ as follows.
\begin{align}\label{def:inn}
	\Inn(*) & = \{ (u\mapsto \lambda^{-1} u\mu,v\mapsto \mu^{-1} v\rho; w\mapsto \lambda^{-1} w\rho) : (\lambda,\mu,\rho)\in \LMR_*^{\times}\}.
\end{align}
In particular $\Inn(*)$ is a central product 
$\Inn(*)  \cong \LL{*}^{\times}\circ \MM{*}^\times \circ \RR{*}^\times \cong (\LMR_*^\times)/\CC{*}^\times$.
We refer to these quotient $\Aut_k(*)/\Inn(*)$ as the ``outer'' autotopisms.
In this way our generalization of Skolem-Neother (Theorem~\ref{thm:Skolem-Noether}) says that 
$k$-linear autotopisms of rectangular matrix products are inner.

\begin{thm}\label{thm:aut-aut}
There is an exact sequence of groups:
\begin{align*}
1\to \Aut_{\LMR}(*) \to \Aut_k(*)\overset{\mathfrak{H}}{\to} \Aut_k(\LMR_*).
\end{align*}  
Furthermore, $(\LMR_*^{\times})\mathfrak{G}\mathfrak{H}=\Inn(*)\mathfrak{H}=\Inn(\LMR_*)$.
\end{thm}
\begin{proof}
First we show that $\Aut_k(*)$ acts on the rings $\LL{*}$, $\MM{*}$, $\RR{*}$, and
$\CC{*}$ by conjugation.  Given $\rho \in \LL{*}$ and $\phi\in \Aut_k(*)$, 
\begin{align*}
	uR_{\phi^{-1}} L_{\rho} R_{\phi} * v 
		& = 	(uR_{\phi^{-1}} L_{\rho} * vR_{\phi^{-1}}) R_{\phi}
		 = 	(u* v)R_{\phi^{-1}} L_{\rho}R_{\phi}.
\end{align*}
Hence, $\rho^{\phi}\in\LL{*}$.
The proof is confirmed similarly for the other rings.
So there is a homomorphism $\mathfrak{H}:\Aut_k(*)\to \Aut(\LMR_*)$.
The kernel of $\mathfrak{H}$ is by definition $\Aut_{\LMR}(*)$.

Take $(\lambda,\mu,\rho)\in \LMR_*^\times$.  Set $\phi=(\lambda,\mu,\rho)\mathfrak{G}$.  Take
$\lambda'\in \LL{*}$.  It follows that
\begin{align*}
	(\lambda')^\phi 
		& = (L^U_{\lambda} L_{\lambda'}^U L^U_{\lambda^{-1}}, L_\lambda^W L_{\lambda'}^W L_{\lambda^{-1}}^W)
		= ( \lambda \lambda'\lambda^{-1}, 1,1).		
\end{align*}
We find similarly in the other components.  So $\phi\mathfrak{H}\in \Inn(\LMR_*)$.
\end{proof}

It remains to prove Theorem~\ref{thm:auto}(E) which asks that we prove
the following exact sequence of groups.
\begin{align}
\tag{E}
1\to Z(\LMR_*)^{\times} \overset{\mathfrak{F}}{\rightarrow} \LMR_*^{\times}\times \Aut_{\LMR}(*) \overset{\mathfrak{G}}{\rightarrow} \Aut_k(*) 
	\overset{\mathfrak{H}}{\rightarrow} \Out_k(\LMR_*).
\end{align}

\begin{proof}[Proof of Theorem~\ref{thm:auto}(E)]
Extend $\mathfrak{G}$ from Theorem~\ref{thm:scalar-aut} to 
\begin{align*}
	\hat{\mathfrak{G}}:\LMR_*^{\times}\times \Aut_{\LMR}(*)\to \Aut_k(*)
\end{align*}
so that
$((\lambda,\mu,\rho),\alpha)\mapsto (\lambda,\mu,\alpha)\mathfrak{G} \alpha^{-1}$.  The kernel of $\hat{\mathfrak{G}}$ 
is parameterized by $(\lambda,\mu,\alpha)$ where $(\lambda,\mu,\rho)\mathfrak{G}\in \Aut_{\LMR}(*)$, i.e. the action by the
scalars on $U$, $V$ and $W$ is $\LMR_*$-linear.  In particular $(\lambda,\mu,\rho)\in Z(\LMR_*^{\times})$ and the converse
is also true.  This defines an exact sequence
\begin{align*}
1\to Z(\LMR_*^{\times})\to \LMR_*^{\times}\times \Aut_{\LMR}(*)\overset{\hat{\mathfrak{G}}}{\to} \Aut_k(*).
\end{align*}
By Theorem~\ref{thm:aut-aut}, $\im \hat{\mathfrak{G}}\mathfrak{H}=\Inn(*)$.  So we complete the exact sequence
(E) by factoring through to the homomorphism
$\bar{\mathfrak{H}}:\Aut_k(*)\to \Out_k(\LMR_*)$. This completes the proof.
\end{proof}

\section{Condensation by Morita equivalence}\label{sec:Morita}
Having reduced the study of autotopisms and derivations to those which commute with $\LMR_*$ we now show how to shrink the
rings $\LMR_*$ to basic subalgebras and correspondingly shrink the original bimap.  Recall that in a ring $A$ an idempotent 
$e\in A$ is {\em full} if $A=AeA$.  We prove Theorem~\ref{thm:auto}(F) along with a version concerning derivations.

\begin{thm}\label{thm:condense}
For a full idempotent $e$ in $\LMR_*$, let $\cond{e}{*}:eUe\times eVe\bmto eWe$ be the restriction of
$*$.  It follows that there are naturally induced isomorphisms
\begin{align*}
	\Aut_{\LMR}(*) & \cong \Aut_{e\LMR e}\left(\cond{e}{*}\right) & \textnormal{and}\quad 
	\Der_{\LMR}(*) & \cong \Der_{e\LMR e}\left(\cond{e}{*}\right).
\end{align*}
\end{thm}

To see the relevance to our problems on isomorphism consider the narrow application to matrix multiplication 
$*:\M_{a\times b}(k)\times \M_{b\times c}(k)\bmto \M_{a\times c}(k)$. The ring
$\LMR_*$ equals $\M_a(k)\oplus \M_b(k)\oplus\M_c(k)$ (Table~\ref{tab:universal-rings}).  
In particular we have a full primitive idempotent, e.g. $e=E_{11}\oplus E_{11}\oplus E_{11}$.  
The resulting condensation of each module $U=\M_{a\times b}(k)$, $V=\M_{b \times c}(k)$, and 
$W=\M_{a\times c}(k)$ is a single copy of $k$.  So $\cond{e}{*}$ maps
$k\times k\bmto k$ and is exactly the product of our field $k$.  The $k$-linear autotopisms
of multiplication in $k$ are just the required two copies of $k^{\times}\times k^{\times}$ given
in Theorem~\ref{thm:scalar-aut}.  By Theorem~\ref{thm:condense}, $\Aut_{\LMR}(*)\cong k^{\times}\times k^{\times}$ as well.

\subsection{Idempotents and direct decompositions}\label{sec:idempotents}

As might be expected, if a bimap is built in a natural way from a ring, module, or group, then a product
in that category will lead to one of the many products of bimaps.  Departure from $*$ to 
look at its surrounding associative rings is a remarkably powerful tool in this case 
used for example to prove results on central and direct products of groups
\citelist{\cite{Wilson:unique-cent}\cite{Wilson:RemakI}*{Section 6}}.
\medskip

To explain the relationship of scalar rings and decompositions, 
recall an element $e\neq 0$ in an associative unital ring $A$ is \emph{idempotent} if $e^2=e$.  
Idempotents $e$ and $f$ are \emph{orthogonal} if $ef=0=fe$.  A \emph{decomposition of $1$} is a set 
$\mathcal{E}$ of pairwise orthogonal idempotents that sum to $1$.  If $A=\End(U)$ and $e\in A$ then
$U=Ue\oplus U(1-e)$ and vice-versa, a direct decomposition of $U$ determines idempotents
in $\End(U)$.  That is how we can involve the scalar rings in the discussion of direct decompositions
of bimaps.  The archetype for this is the study of self-adjoint idempotents associated to 
nondegenerate bilinear form.  Applying this to general bimaps is done in work of 
Miyasnikov \cite{Myasnikov} and the author \citelist{\cite{Wilson:unique-cent}*{Section 4}\cite{Wilson:RemakI}*{Section 6}}.

\begin{prop}
Fix a $k$-bimap $*:U\times V\bmto W$.
\begin{enumerate}[(i)]
\item A decomposition $\mathcal{E}\subset \End(U_k)\times \End(W_k)$ of $1$ lies in 
$\LL{*}$ if, and only if,
\begin{align*}
	(\forall & e\in \mathcal{E}) & (eU) * V & \leq eW.
\end{align*}
In that case $*$ admits a left $\oplus$-decomposition
$\{ eU\times V\bmto eW : e\in \mathcal{E}\}$.

\item A decomposition $\mathcal{E}\subset \End(_k U)\times \End(V_k)$ of $1$ lies in 
$\MM{*}$ if, and only if,
\begin{align*}
	(\forall & e\in \mathcal{E}) & (Ue) * ((1-e) V) & = 0.
\end{align*}
In that case $*$ admits a $\perp$-decomposition
$\{ Ue\times eV\bmto W : e\in \mathcal{E}\}$.

\item A decomposition $\mathcal{E}\subset \End(_k V)\times \End(_k W)$ of $1$ lies in 
$\RR{*}$ if, and only if,
\begin{align*}
	(\forall & e\in \mathcal{E}) & U*(Ve) & \leq We.
\end{align*}
In that case $*$ admits a right $\oplus$-decomposition
$\{ U\times Ve\bmto We : e\in \mathcal{E}\}$.
\end{enumerate}
\end{prop}
\begin{proof}
We prove (i). Fix $e\in \End(U_K)\oplus \End(W_K)$. 

Suppose that $U=eU\oplus (1-e)U$ and $W=eW\oplus (1-e)W$ 
such that $eU*V\leq eW$ and $(1-e)U*V\leq (1-e)W$.
For every $u\in U$ and $v\in v$, $(eu)*v \in (eU)*V\leq eW$ and so $e((eu)*v)=(eu)*v$.
Likewise $((1-e)u)*v\in (1-e)W$ which implies $e((1-e)u)*v=0$. Furthermore, $e(u*v)=e((eu)*v)+e((1-e)u)*v)=(eu)*v$.
That is, $(L_e^U, L_e^W)\in \LL{*}$.

Conversely, if $e\in \LL{*}$ is an idempotent.  Then $(eu)*v=e(u*v)$ so $(eU)*V\leq eW$ and likewise $((1-e)U)*V\leq (1-e)W$.

A proof of (iii) is similar, a proof of (ii) can also be adapted from this argument; cf. \cite{Wilson:unique-cent}*{Proposition 4.7}.
\end{proof}

\begin{coro}
A bimap $*$ is left, mid, or right indecomposable if, and only if, $\LL{*}$, $\MM{*}$, resp. $\RR{*}$ is a local ring.
\end{coro}

Arguing with the centroid we arrive at the following.

\begin{prop}[Myasnikov \cite{Myasnikov}*{Proposition 3.1}]
A $k$-bilinear map $*:U\times V\bmto W$ admits a $\oplus$-decomposition $\{Ue\times V\bmto We : e\in \mathcal{E}\}$
for a decomposition $\mathcal{E}\subset \End({_k U}_k)\times \End({_k V}_k)\times \End({_k W}_k)$ of $1$ if, and only if,
$e\in \CC{*}$.
\end{prop}

\subsection{Change of scalars}\label{sec:change-scalars}

We now explore a more constrained role for idempotents with stronger connection to our main results.
We begin with a natural construction of bimaps from old ones.  Fix an $\LMR$-bimap $*:U\times V\bmto W$.  
\medskip

For every $(\LL{}',\LL{})$-bimodule $X$ we can define an
$(\LL{}',\MM{},\RR{})$-bimap $(X\otimes *): X\otimes_{\LL{}} U\times V\bmto X\otimes_{\LL{}} W$
where
\begin{align*}
		\left(\sum_i x_i\otimes u_i\right)(X\otimes *) v & = \sum_i x_i\otimes (u_i* v).
\end{align*}
This extends to a covariant functor from the homotopism category of left $\LL{}$-linear bimaps to 
that of left $\LL{}'$-linear bimaps.  Specifically a homotopism $(f,g;h)$ becomes $(1_X\otimes f,g;1_X\otimes h)$.

Similarly, given an $(\LL{},\LL{}')$-bimodule $X'$,  define 
$(X'\lversor *):X'\lversor_{\LL{}} U\times V\bmto X'\lversor_{\LL{}} W$,
\begin{align*}
	(\forall & \phi:X'\to U,\forall v\in V) & 
		\phi (X'\lversor *) v & : x\mapsto (x\phi*v).
\end{align*}
This is an $({\LL{}'}^{\circ},\MM{},\RR{})$-bimap.  We have elected here to define these
bimaps according to the fixed representations of tensor and versor products so that the products
are given explicitly.  It is possible to give
definitions that rely solely on the universal properties of tensor and versor products instead.

The preceding discussion adapts to explain the meaning of $*\otimes Y$, for $Y={_{\RR{}} Y_{\RR{}'}}$, 
and $*\rversor Y'$ for $Y'=*\rversor {_{\RR{}'}Y_{\RR{}}}$.  The mid variations take a mixed from.  
Here we have an $(\MM{},\MM{}')$-bimodule $Z$ and define an $(\LL{},\MM{}',\RR{})$-bimap
\begin{align*}
U\otimes_{\MM{}} Z\times Z\lversor_{\MM{}} V & \bmto W\\
\left(\sum_i u_i \otimes z_i, \phi:Z\to V\right) & \mapsto \sum_i u_i\otimes (z_i \phi).
\end{align*}
Likewise each $(\MM{}',\MM{})$-bimodule $Z'$ affords 
$U\rversor_{\MM{}} Z'\times Z'\otimes_{\MM{}} V\bmto W$.

As with our first examples, these all extend to a covariant functors.  Though it is easiest to explain these
individually they can be used in conjunction as well.  For example, if we begin with the product 
$*:A\times A\bmto A$ of a unital associative ring $A$ then $\LL{*}\cong \MM{*}\cong\RR{*}\cong A$.
Then we can tensor the left by $A^{a\circ}=A^a\lversor A$, i.e. $1\times a$ column vectors, 
tensor the right by $A^{c\circ}$, and in the middle use $A^{b}$.  We so obtain
a the familiar bimap of matrix multiplication over $A$, i.e.
\begin{align}\label{eq:inflate}
	\M_{a\times b}(A)\times \M_{b\times c}(A)\bmto \M_{a\times c}(A).
\end{align}
This is seen naturally as an $(\M_a(A),\M_b(A),\M_c(A))$-bimap.  We might instead consider condensing this
product into something smaller.  By tensoring with $A^{a}$, $A^{b\circ}$, and $A^{c}$ we return from 
\eqref{eq:inflate} to $A\times A\bmto A$.  Such a change in bimaps mimics our experience with Morita
equivalence, and the use of basic rings in representations of finite-dimensional algebras.

\subsection{Condensation}\label{sec:condensation}
We have seen in \eqref{eq:inflate} how to inflate and condense the product of an associative unital ring into
matrix multiplication of a general sort.  Using the case of $a=b=c$ this is nothing more than the standard example
of $A$ being Morita equivalent to $\M_a(A)$.  We now want such a result in general bimaps.  In particular
we want to consider the group $\Aut_{\LMR}(*:U\times V\bmto W)$ in relation $\Aut(eUe\times eVe\bmto eWe)$
for idempotents $e\in \LMR$. This will permit us to mimic the role that basic subalgebras play in representation theory.
If nothing else we get  smaller rings and modules to study.

We invoke the well-known Morita equivalence theorem, cf. \cite{Rowen:II}*{Theorem 25A.19}.
\begin{thm}[Morita]
Two rings $\LL{}$, $\LL{}'$ have equivalent module categories if, and only if, there is an $\LL{}$-progenerator
$P={_{\LL{}} P}$ with $\LL{}'\cong \End(_{\LL{}} P)$.  In particular an equivalence is afforded by
the pair $P^{\circ}\otimes_{\LL{}} (-)$ and $P\otimes_{\LL{}'} (-)$, where $P^\circ=(P\;{_{\LL{}}\lversor}\; \LL{})$.
\end{thm}

We note that our convention of evaluating opposite scalars means that in our statement above $P$ is an
$(\LL{},\LL{}')$-bimodule, without appeal to $op$-notation as is common in some treatments of Morita equivalence.

\begin{prop}
Given a Morita equivalence between $\LL{}$ and $\LL{}'$ afforded by an $\LL{}$-progenerator $P$, 
every left $\LL{}$-bimap $*:U\times V\bmto W$ is naturally isotopic to the induced bimap 
$P \otimes P^{\circ}\otimes U\times V\bmto P\otimes P^{\circ}\otimes W$. In fact the isotopism is the
identity on $V$.
\end{prop}
\begin{proof}
By Morita's theorem, for each $U={_{\LL{}} U}$, there are natural isomorphisms $\tau_U:P\otimes P^{\circ}\otimes U\to U$, specifically
\begin{align*}
	\sum_i p_{i}\otimes f_i\otimes u_i & \mapsto \sum_i (p_i f_i) u_i.
\end{align*}
We verify $(\tau_U,1_V;\tau_W)$ is an isotopism from 
$\#:P\otimes (P^{\circ}\otimes U)\times V\bmto P\otimes (P^{\circ}\otimes W)$ to
$*:U\times V\bmto W$.
\begin{align*}
	\left(\sum_{i} p_i \otimes \left(\sum_j f_{ij} \otimes u_{ij}\right)\right)\tau_U *v 
		& = \sum_{ij} p_i f_{ij} u_{ij} *v 
		 = \left(\sum_{ij} p_i\otimes f_{ij}\otimes (u_{ij} *v)\right)\tau_W\\
		& = \left(\sum_i p_i\otimes \left(\sum_j f_{ij}\otimes u_{ij}\right)\right)\#^{\tau_W} v. 
\end{align*}
\end{proof}

The usual occasion for such equivalences is to specify an idempotent $e\in \LL{}$ which is \emph{full} in that $\LL{}=\LL{} e\LL{}$.
We obtain a Morita equivalence, i.e. an equivalence of module categories, between $e\LL{}e$ and $\LL{}$.
Specifically, $P=\LL{}e$ and $P^{\circ}\cong e\LL{}$ so that $P^{\circ}\otimes_{\LL{}} U=e\LL{}\otimes_{\LL{}} U\cong eU$ 
and $P\otimes_{e\LL{} e} eU \cong \LL{}e\otimes_{e\LL{} e} eU\cong \LL{}e\LL{}U=U$.
More generally, for a full idempotent $e\in \LMR$, we get a \emph{condensed}
$e\LMR e$-bimap $eUe\times eVe\bmto eWe$.

\begin{proof}[Proof of Theorem~\ref{thm:auto}(F) \& Theorem~\ref{thm:condense}]
Let $\phi\in \Aut_{\LMR}(*:U\times V\bmto W)$ and $e=((L_e^U,L_e^W), (R_e^U,L_e^V), (R_e^V,R_e^W) )$ be a full idempotent in $\LMR$.
Define
\begin{align*}
	\phi|_e & = \left((L_e^U \phi_U R_e^U)|_{eUe}, (L_e^V \phi_V R_e^V)|_{eVe}; (L_e^W \phi_W R_e^W)|_{eVe}\right).
\end{align*}
As $\phi$ is $\LMR$-linear, $L_e^U \phi_U R_e^U=L_e^U R_e^U\phi_U=\phi_U L_e^U R_e^U$, and likewise in the variables
$V$ and $W$.  Hence, $\phi|_e\in \Aut_{e\LMR e}(eUe\times eVe\bmto eWe)$.  Also, $(\phi\psi)|_e=\phi|_e\psi|_e$ as, for example,
\begin{align*}
	L_e^U \phi_U\psi_U R_e^U & = L_e^U L_e^U \phi_u \psi_u R_e^U R_e^U
		(L_e^U \phi_U R_e^U) (L_e^U \psi_U R_e^U).
\end{align*}
So $\phi\mapsto \phi|_e$ is a group homomorphism $\Aut_{\LMR}(*:U\times V\bmto W)$ to $\Aut_{e\LMR e}(eUe\times eVe\bmto eWe)$.
For the inverse homomorphism, take $\phi'\in \Aut_{e\LMR e} (eUe\times eVe \bmto eWe )$ and define
\begin{align*}
	\phi'|^e & = ( 1_{\LL{}e}\otimes \psi_{eUe}\otimes 1_{e\MM{}}, 
				1_{\MM{}e}\otimes \psi_{eVe}\otimes 1_{e\RR{}}; 
					1_{\LL{}e}\otimes \psi_{eWe}\otimes 1_{e\RR{}} ).
\end{align*}
Indeed, $(\phi'|^e)|_e=\phi'$ since, for instance in the $U$ variable, we find
\begin{align*}
	(L_e^U 1_{\LL{}e}\otimes \psi_{eUe}\otimes 1_{e\MM{}} R_e^U)|_{eUe} & = \psi_{eUe}.
\end{align*}

The proof of $\Der_{\LMR}(*)\cong \Der_{e\LMR e}(e*e)$ is analogous.
\end{proof}

\section{A theorem of Skolem-Noether type}\label{sec:universal}

We now apply the main results to universal bimaps (tensor and versor products)
and generalize the Skolem-Noether theorem from square matrix rings to arbitrary matrix multiplication.

\subsection{Tensor and versor products}
Here we consider the implications for the case of tensor and versor products.  

We start by explaining a Galois closure on scalar rings.  Originally discovered for
mid scalars in \cite{Wilson:division}*{Theorem 2.8}, it was proved fully in \cite{BW:autotopism}*{Section 2.2}.
Adapting those arguments to left and right scalars is possible but somewhat tangential; 
so, we give simply the result of the claims.  We start with the following definitions of closures.
\begin{align*}
	\bar{\LL{}}  & = \LL{_{\LL{}}\lversor} &
		\bar{\MM{}}  & = \MM{\otimes_{\MM{}}} &
			\bar{\RR{}}  & = \RR{_{\rversor\RR{}}}.
\end{align*}

\begin{prop}
\begin{align*}
	\bar{\bar{\LL{}}} & =\bar{\LL{}} & 
		\bar{\bar{\MM{}}} & =\bar{\MM{}} & 
			\bar{\bar{\RR{}}} & =\bar{\RR{}}.\\
	{_{\LL{}} \lversor} & ={_{\bar{\LL{}}}\lversor} &
		\otimes_{\MM{}} & =\otimes_{\bar{\MM{}}} & 
			\rversor_{\RR{}} & =\rversor_{\bar{\RR{}}}.
\end{align*}
In particular all versor and tensor products can be taken over closed scalar rings.
\end{prop}

Before considering our main theorems for tensor and versor products we pause to remark that
these products can be degenerate, for example, $\mathbb{Z}_2\times \mathbb{Z}_3\bmto \mathbb{Z}_2\otimes \mathbb{Z}_3$ or 
$\mathbb{Z}_2 \times \mathbb{Z}_2\lversor \mathbb{Z}\bmto \mathbb{Z}$ will both evaluate to $0$.  
We continue to consider solely the nondegenerate case so exclude such examples in our discussion below.

Given that we need only look at tensor and versor products over closed rings we can in fact describe
the invariants explicitly (assuming fully nondegenerate products).  See Table~\ref{tab:universal-rings}.
To verify this notice for instance in the case of ${_{\LL{}} \lversor}:U\times U _{\LL{}}\lversor W\bmto W$,
$\End(_{\LL{}} W)$ already acts on $W$ and also in a natural way on 
$U_{\LL{}}\lversor W=\hom(_{\LL{}} U,_{\LL{}} W)$.  So by the universality of the right scalars $\End(_{\LL{}}W)$
embeds in $\RR{_{\LL{}}\lversor}$.  But by a 3-pile-shuffle and the assumption of nondegeneracy, 
we need to ensure $W$ is an $(\LL{},\RR{_{\LL{}}\lversor})$-bimodule
so this limits $\RR{_{\LL{}}\lversor}$ to be isomorphic to $\End(_{\LL{}} W)$.  Similar claims explain the other entries.

\begin{table}[!htbp]
\begin{tabular}{|c|ccc|c|c|}
\hline
$*$	&  $\LL{*}$ & $\MM{*}$ & $\RR{*}$ & $\CC{*}$ & $\Aut(*)$\\
\hline\hline
$_{\LL{}}\lversor$ & $\LL{}$ & $\End(_{\LL{}} U)$ & $\End(_{\LL{}} W)$ & $Z(\LL{})$ & $N_{\Aut(_{\LL{}} U)\times\Aut(_{\LL{}} W)}(\LL{})$\\
$\otimes_{\MM{}}$ & $\End(U_{\MM{}})$ & $\MM{}$ & $\End(_{\MM{}} V)$ & $Z(\LL{})$ & $N_{\Aut(U_{\MM{}})\times\Aut(_{\MM{}} V)}(\MM{})$\\
$\rversor_{\RR{}}$ & $\End(W_{\RR{}})$ & $\End(V_{\RR{}})$ & $\RR{}$ & $Z(\LL{})$ & $N_{\Aut(V_{\RR{}})\times\Aut(W_{\RR{}})}(\RR{})$\\
\hline
\end{tabular}
\caption{Given closed rings $\LL{}=\bar{\LL{}}\subseteq \End(U_K)\oplus \End(W_K)$, $\MM{}=\bar{\MM{}}\subseteq \End(_K U)\oplus \End(V_K)$, 
and $\RR{}=\bar{\RR{}}\subseteq \End(_K V)\oplus \End(_K W)$, the remaining universal scalars
are determined as above.}\label{tab:universal-rings}
\end{table}

We can further explain the autotopisms of tensor and versor products as normalizers.  This observation was made
in the case of tensors in \cite{BW:autotopism}*{Theorem 1.5} but we now see it as a general consequence.
For example, let $\phi=(L_{\phi}^U,L_{\phi}^W)$ normalize $\LL{}=\bar{\LL{}}$.  For $f:U\to W\in U_{\LL{}}\lversor W$, define
$fR_\phi^{U\lversor W}=L_{\phi^{-1}}^U f L_{\phi}^W$.
\begin{align*}
	u\phi\lversor f\phi & = uL_{\phi}^U L_{\phi^{-1}}^U fL_{\phi}^W = (u\lversor f)\phi.
\end{align*}
So $(L_{\phi^{-1}}^U,R_\phi^{U\lversor W}; L_{\phi}^W)\in \Aut(\lversor)$.  This explains the last column
in Table~\ref{tab:universal-rings}.  Of course it is far more descriptive to apply the exact sequences of Theorem~\ref{thm:auto}.

We close by demonstrating how to use these results when we have well-behaved rings, proving Theorem~\ref{thm:Skolem-Noether}.

\begin{thm*}[Generalized Skolem-Noether]
For a field $k$, the $k$-linear autotopisms and derivations of matrix multiplication
$*:\M_{a\times b}(k)\times \M_{b\times c}(k)\bmto \M_{a\times c}(k)$ are inner.
\end{thm*}
\begin{proof}
First observe that $\LL{*}\cong \M_a(k)$, $\MM{*}\cong \M_b(k)$, $\RR{*}\cong \M_c(k)$, and $\CC{*}\cong k$; 
cf. Table~\ref{tab:universal-rings}.  By the Skolem-Noether theorem $\Out_K(\M_a(k))\cong \Out_K(\M_b(k))\cong \Out_K(\M_c(k))\cong 1$.
By Theorem~\ref{thm:auto} we therefore have the following exact sequence.
\begin{align*}
	1\to (k^\times)^3 \to \GL_a(k)\times \GL_b(k)\times \GL_c(k)\times \Aut_{\LMR}(*)\to \Aut_k(*)\to 1.
\end{align*}
Fix a primitive idempotent $e$.  By Theorem~\ref{thm:condense}, $\Aut_{\LMR}(*)\cong \Aut_{e\LMR e}(e*e)$.
Since $e$ is primitive, and full, $e\LMR e\cong k$ and $e\M_{a\times b}(k)e\cong k$, $e\M_{b\times c}(k)e\cong k$,
and $e\M_{a\times c}(k)e\cong k$.  So $\Aut_{e\LMR e}(e*e)=\Aut_k(k\times k\bmto k)\cong k^\times \times k^\times$.
The result now follows.
\end{proof}

%
\section{Claims under symmetry and degeneracy}\label{sec:sym-deg}

Already in the study of $\Aut(\mathcal{T}_{abc}(k))$ it becomes necessary to consider autotopisms
that preserve symmetry and bimaps that are degenerate.  So in this section we adapt the methods above
in these two ways.

\subsection{Structure of pseudo-isometries}
Now we consider a special context that is prevalent to algebra.  Often a bimap $*:U\times V\bmto W$ is
symmetric, alternating, or in general weakly Hermitian in the sense we defined in Section~\ref{sec:prelims}.
So throughout this section suppose that $*:U\times V\bmto W$ is fully nondegenerate and weakly Hermitian
with respect to $\tau$, i.e. $\tau:*\to \tilde{*}$ is an isotopism and $\tau\tilde{\tau}=1_*$.  In particular
$R^V_{\tau}=(R_{\tau}^U)^{-1}$ and $(R_{\tau}^W)^2=1_W$.  Recall that for $\alpha\in \End(_k X)$,
$R_{\alpha}^X\mapsto L_{\alpha}^X$ is our anti-isomorphism $\End(_k X)\to \End(X_k)$.

First observe that $\tau$ induces anti-isomorphisms $\LL{*}\to \RR{*}$, $\MM{*}\to \MM{*}$, and
$\RR{*}\to \LL{*}$ as follows.
\begin{align*}
	\bar{\lambda} & = \overline{(L_{\lambda}^U, L_{\lambda}^W)} 
		=(R^V_{\tau}R_{\lambda}^UR^U_{\lambda},R^W_{\tau}R^W_{\lambda}R^W_{\tau})\\
	\bar{\mu} & = \overline{(R_{\mu}^U,R_{\mu}^W)} 
		= (R_{\tau}^U R^V_{\mu}R_{\tau}^V, R_{\tau}^V L_{\mu}^UR_{\tau}^V),\\
	\bar{\rho} & = \overline{(R_{\rho}^V,R_{\rho}^W)} 
		= (R_{\tau}^U R_{\rho}^V R_{\tau}^V, R_{\tau}^W R_{\rho}^W R^W_{\tau}).
\end{align*}
In particular there is a ring involution on $\LMR_*$ given by
\begin{align}\label{def:involution}
	\overline{(\lambda, \mu,\rho)} = (\bar{\rho},\bar{\mu},\bar{\lambda}).
\end{align}
Given an algebra $A$ with involution $a\mapsto \bar{a}$ we define:
\begin{align*}
	U(A,\bar{\cdot}) & = \{ a\in A^{\times} : a\bar{a}=1=\bar{a}a\},\\
	\Aut(A,\bar{\cdot}) & = \left\{ \phi\in \Aut(A) : \bar{a}\phi=\overline{a\phi}\right\},\\
	\Inn(A,\bar{\cdot}) & = \{\phi\in \Inn(A) : \bar{a}\phi=\overline{a\phi}\},\\
	Z(A,\bar{\cdot}) & = \{ z\in Z(A) : a^*=a\}.
\end{align*}

If we follow the details in the proof of Theorem~\ref{thm:auto} we see:
\begin{thm}\label{thm:pseudo}
\begin{align*}
	1\to Z(\LMR_*^{\times},\bar{\cdot})\to U(\LMR_*,\bar{\cdot})\times \Psi\Isom_{\LMR}(*) \to \Psi\Isom(*)\to\frac{\Aut(\LMR_*,\bar{\cdot})}{\Inn(\LMR_*, \bar{\cdot})}.
\end{align*}
\end{thm}

The group $U(\LMR_*,\bar{\cdot})$ is not as easy to describe as the units of a ring; however, using work of Weil, Wagner, Taft,
and Brooksbank and the author there is now a rather robust understanding of this group and indeed a polynomial time algorithm
to compute the group.  See \cite{BW:isom} for details and full bibliography.  

\subsection{The weakly-Hermitian matrix products}
Given our interest in automorphisms of $\mathcal{T}(*;S)$ for $S=\{0\}$ and $S=\{1\}$ we need to consider a specific
weakly-Hermitian bimap.  Given a bimap $*:U\times V\bmto W$ and $\epsilon=\pm 1$, define $\#:(U\oplus V)\times (V\oplus U)\to W$ as
\begin{align*}
	(u,v)\# (v',u') & = u*v'+ \epsilon (u'*v).
\end{align*}
This is an orthogonal sum of $*$ and $\epsilon\tilde{*}$.  As a consequence its properties can be derived
from $*$.  Furthermore, observe that $\#$ is weakly $\tau$-Hermitian where $(u,v)R_{\tau}^U=(v,u)$, $(v,u)R_{\tau}^V=(u,v)$,
and $wR^W_{\tau}=-w$.  Our interest will be in the pseudo-isometries of $\#$ as a function of $*$.  Following Theorem~\ref{thm:pseudo}
we have reduced the question to computing $U(\LMR_{\#},\bar{\cdot})$.  This work is to describe $\LMR_{\#}$ as a solution to 
a system of equations, then define the unitary elements under the involution.  For any specific choice of $(a,b,c)$ and $k$
there is an efficient algorithm for the task given in \cite{BW:isom}.  The purpose of the computations in this section
are to resolve the problem for all $(a,b,c)$, and since the approach is essentially linear algebra this is indeed nothing more
than a symbolic computation.  To make it slightly quicker in places we apply some shortcuts using Morita condensation.

We need to briefly describe the category of Adjoints of a bimap as introduced in \cite{Wilson:division} and found independently in
\cite{BFFM}.

Fix $W$.  Given bimaps $*:U_*\times V^*\bmto W$ and $\bullet:U_{\bullet}\times V^{\bullet}\bmto W$, an {\em adjoint-morphism}
is a pair $\phi=(R^U_\phi:U_*\to U_{\bullet},L_\phi^V:V^{\bullet}\to V^*)$ such that
\begin{align*}
	(\forall & u\in U_*, \forall v'\in V^{\bullet}) & u\phi \bullet v' & = u*\phi v'. 
\end{align*}
This forms a category, indeed an abelian category, and it has many useful properties similar to categories of modules.
These are given in detail in \citelist{\cite{BFFM}\cite{Wilson:division}}.  What we need is to observe that the
hom-sets in this category are natural abelian groups, and to avoid confusion with homotopism categories we denote
them $\Adj(*,\bullet)$.  Notice that for a fixed bimap $*$, $\Adj(*,*)=\MM{*}$.  This invites a further universal perspective
to the description of the rings $\MM{*}$.  (As might be assumed by this point, there are also natural categories in which
the rings $\LL{*}$ and $\MM{*}$ serve as the endomorphisms.)

\begin{prop}\label{prop:twist}
$\CC{\#}\cong \LL{\#}\cong \RR{\#}\cong \CC{*}$ and
\begin{align*}
	\MM{\#} & = \left\{ \left(\begin{bmatrix} R_\mu^U & R_{\rho}^U \\ R_{\psi}^V & R_{\nu}^V \end{bmatrix},
						\begin{bmatrix} L_\mu^V & \epsilon L_{\rho}^V \\ \epsilon L_{\psi}^U & L_{\nu}^U \end{bmatrix}\right) :
						\begin{array}{cc}
							\mu\in \MM{*}, &  \rho\in \Adj(*,\epsilon \tilde{*}),\\ 
							\psi\in \Adj(\epsilon \tilde{*},*), & \nu\in \MM{\tilde{*}}
						\end{array}\right\}.
\end{align*}
\end{prop}
\begin{proof}
Define $\CC{*}\to \CC{\#}$ by $\sigma\mapsto \left(R_{\sigma}^U \oplus R_{\sigma}^V, R_{\sigma}^V\oplus R_{\sigma}^U; R_{\sigma}^W\right)$.
This is a ring embedding.  Accordingly we have embeddings $\CC{*}\to \LL{\#}$ and $\CC{*}\to \RR{\#}$.  We now claim
each is an isomorphism.  We prove this for $\RR{\#}$ and remark the others follow similarly.

Fix $\Sigma\in \RR{\#}$ and decompose $R_{\Sigma}^{U_{\#}}=\begin{bmatrix}\Sigma_{11} & \Sigma_{12} \\ \Sigma_{21} & \Sigma_{22} \end{bmatrix}$ with respect 
to the decomposition $V_{\#}=V\oplus U$.  Then for all $u\in U$, $v'\in V$,  
\begin{align*}
	(u*v')R_{\Sigma}
	((u,0)\#(v',0))R_{\Sigma}^W & = (u,0) \# (v',0)\begin{bmatrix}\Sigma_{11} & \Sigma_{12} \\ \Sigma_{21} & \Sigma_{22} \end{bmatrix}
		 = u*v'\Sigma_{11}.
\end{align*}
Hence, $(\Sigma_{11},R_{\Sigma}^W)\in \RR{*}$.  By instead using the restriction to $0\oplus V\times 0\oplus U\bmto W$ we find
$(\Sigma_{22},R_{\Sigma}^W)\in \RR{\epsilon\tilde{*}}=\LL{*}$.  This means
\begin{align*}
	u\Sigma_{22}*v' & = (u*v') R_{\Sigma}^W = u*(v'\Sigma_{11}).
\end{align*}
That is to say, $(\Sigma_{22},\Sigma_{11};R_{\Sigma}^W)\in \CC{*}$.  Since $\CC{*}$ embeds in $\RR{\#}$ in just this way we can subtract off
the block diagonal of $\Sigma$ and be concerned solely with the remaining case where $\Sigma_{11}=0$ and $\Sigma_{22}$.  Here we consider the
restrictions to $U\oplus 0\times 0\oplus V\bmto W$ and $0\oplus V\times V\oplus 0\bmto W$.  As $U\oplus 0$ and $0\oplus V$ are isotropic it
follows that for all $u\in U$, $0=U* u\Sigma_{12}$ so that $u\Sigma_{12}\in U^{\bot}=0$.  Likewise $\Sigma_{21}=0$.  Therefore the embedding
$\CC{*}\to \RR{\#}$ is an isomorphism.

To compute the structure of $\mu\in \MM{\#}$ write
\begin{align*}
	R_{\mu}^{U_{\#}} &  = \begin{bmatrix} F_{11} & F_{12}\\ F_{21} & F_{22} \end{bmatrix}
	&	L_{\mu}^{V_{\#}} &  = \begin{bmatrix} G_{11} & G_{12}\\ G_{21} & G_{22} \end{bmatrix}
\end{align*}
with respect to $U\oplus V$ and $V\oplus U$.
By restriction to $U\oplus 0\times V\oplus 0$ shows that $(F_{11},G_{11})\in \Adj(*,*)=\MM{*}$.  Likewise 
$(F_{22},G_{22})\in \MM{\tilde{*}}$.  For the crossed term $(F_{12},G_{12})$ takes $U\oplus 0$ to 
$0\oplus V$, which is orthogonal to $V\oplus 0$.  Therefore the constraints on $(F_{12},G_{12})$ 
are determined by the constraint of $\#$ to $U\oplus 0\times 0\oplus U$ and so lies in 
$\Adj(*,\tilde{*})$.  The rest follows similarly.
\end{proof}

\begin{ex}\label{ex:adj-mat}
For $(a,b,c)$-matrix multiplication $*:\M_{a\times b}(k)\times \M_{b\times c}(k)\bmto \M_{a\times c}(k)$,
\begin{align*}
	\Adj(*,\epsilon\tilde{*}) & = \left\{\begin{array}{cc} 	
		\M_b(k)^{\oplus c}, & a=1;\\
		0, & a>1.\end{array}\right.
&	\Adj(\epsilon\tilde{*},*) & = \left\{\begin{array}{cc} 		
		\M_b(k)^{\oplus a}, & c=1;\\
		0, & c>1.\end{array}\right.
\end{align*}
\end{ex}
\begin{proof}
As adjoints of $k$-bimaps are an abelian category the hom-sets are determined by a system of $k$-linear
equations.  Using the matrix units $E_{ij}$ the constraints for $\Adj(*,\tilde{*})$ can be computed.
A less arduous approach is to apply Morita condensation.  

First condense form
$(a,b,c)$-matrix products to $(a,1,c)$-matrix products $*:\M_{a\times 1}(k)\times \M_{1\times c}(k)\bmto \M_{a\times c}(k)$.  
In particular $\tilde{*}$ is isotopic to $(c,1,a)$-matrix multiplication.  

If $a=1$ then we are comparing the left $k$-vector space $k\times k^c\bmto k^c$ with the right 
$k$-vector space $\tilde{*}:k^c\times k\bmto k^c$.  
Given $F\in \M_{1\times c}(k)$ it follows that
$s,t\in k$, $(sF)t=s(\epsilon^2 Ft)$ so that $(F,\epsilon F)\in \Adj(*,\epsilon\tilde{*})$.
As $*$ is nondegenerate given any $(F,G)\in \Adj(*,\tilde{*})$, $G$ is determined by $F$.  In particular
$\Adj(*,\tilde{*})\cong \M_{1\times a}(k)$.  Then we re-inflate using $\otimes \M_b(k)$ to 
find that for $(1,b,c)$-matrix products $*$, $\Adj(*,\epsilon\tilde{*})\cong k^a \otimes \M_b(k)$.
By contrast if  we look at $\Adj(\epsilon\tilde{*},*)$ then we need $F\in \M_{c\times 1}(k)$ 
and $G\in \M_{1\times c}(k)$ such for $u,v\in k^c$, $t\in k$, $(u F)v^t=u (Gv^t)$.  Unless $F$ is 
square this cannot be, and $F$ is square only if $c=1$. 

If instead $a>1$ then we are considering the tensor products $*:k^a\times k^c\to k^a\otimes_k k^c$
and $\tilde{*}:k^{c}\times k^a\bmto k^c\otimes_k k^a$.  Here the adjoint-morphisms are always trivial.
\end{proof}

\begin{thm}
Given the bimap $*:\M_{a\times b}(k)\times \M_{b\times c}(k)\bmto \M_{a\times c}(k)$ and
$\# = * \bot (\epsilon \tilde{*})$, it follows that
\begin{align*}
	\Psi\Isom(\#) & = \left\{ \begin{array}{cc}
		({\rm Sp}_{2b}(k)\times k^{\times})\rtimes \Gal(k) &  a=c=1\\
		(\GL_b(k)\times k^{\times})\rtimes \Gal(k) & 1<a,c\\
		\left(k^a\otimes \M_b(k)\otimes k^c \rtimes (\GL_b(k)\times k^{\times})\right)\rtimes \Gal(k) & \textnormal{else}.\\
	\end{array}\right.
\end{align*}
\end{thm}
\begin{proof}
Following Proposition~\ref{prop:twist}  we know 
$\LL{\#}=k$ and $\RR{\#}=k$.  Now for $\MM{\#}$ we have four cases.  

First if $a=c=1$ then $\#$ is then $\MM{\#}=\M_2(\M_b(k))$ with the adjugate involution
\begin{align*}
\overline{\begin{bmatrix}
A & B\\ C & D 
\end{bmatrix}}=\begin{bmatrix}
D^t & -B^t\\ -C^t & A^t
\end{bmatrix} 
= \begin{bmatrix} 0 & -I_b \\ I_b & 0 \end{bmatrix}
\begin{bmatrix}
A & B\\ C & D 
\end{bmatrix}^t
\begin{bmatrix} 0 & I_b \\ - I_b & 0 \end{bmatrix}.
\end{align*}
From this it follows that $U(\MM{\#})$ is isomorphic to the symplectic group ${\rm Sp}_{2b}(k)$.
So $U(\LMR_{\#})=\frac{k^{\times }\times \Sp_{2b}(k)\times k^{\times}}{\langle (s,1,s^{-1}) : s\in k\rangle}$.
Next, an outer automorphism of $\LMR_{\#}=k\oplus M_{2b}(k)\oplus k$ that commutes with the involution
$(\lambda,\mu,\rho)\mapsto(\rho,\bar{\mu},\lambda)$ must fix all three direct factors.  By the classic Skolem-Noether
theorem it follows that there are no $k$-linear outer automorphism.  As a result
$\Aut(\LMR_{\#},\bar{\cdot})/\Inn(\LMR_{\#},\bar{\cdot})\cong \Gal(k)$.  Lastly, Morita condensation
of $\#$ reduces $U\oplus V=\M_{1\times b}(k)\oplus \M_{b\times 1}(k)\cong k^{2b}$ to $k$.  
So then $\Psi\Isom_{\LMR}(\#)\cong \Psi\Isom(k\times k\bmto k)\cong k^{\times}$.  As $Z(\LMR_{\#}, \bar{\cdot})\cong k^{\times}$
it follows from Theorem~\ref{thm:pseudo} that
\begin{align*}
1\to k^{\times} & \to (k^{\times} \times {\rm Sp}_{2b}(k))\times k^{\times}\to \Psi\Isom(\#)\to \Gal(k).
\end{align*}
So ${\rm Sp}_{2b}(k)\times k^{\times}\leq \Psi\Isom(\#)\leq ({\rm Sp}_{2b}(k)\times k^{\times})\rtimes \Gal(k)$.  
By inspection we confirm $\Psi\Isom(\#)=({\rm Sp}_{2b}(k)\times k^{\times})\rtimes \Gal(k)$.\footnote{Alternatively 
observe that when $a=c=1$ the product $\#:k^{2b}\times k^{2b}\bmto k$
is none other than the an alternating nondegenerate $k$-form.  Hence, $U(\MM{\#})={\rm Sp}_{2b}(k)$ by
definition. In fact this is the case of a Heiseberg group/Lie algebra where the structure of
$\Psi\Isom(\#)={\rm \Gamma SP}_{2b}(k)$ is classically described.}

The case $a,c>1$ is nearly identical except that the constituent rings change in structure.  By 
Example~\ref{ex:adj-mat} $\Adj(*,\epsilon \tilde{*})$ and $\Adj(\epsilon\tilde{*},*)$ are trivial and
so $\MM{\#}=\MM{*}\oplus \MM{\tilde{*}}\cong \M_b(k)\oplus \M_b(k)$ with involution interchanging the two simple factors.
Thus, $U(\MM{\#})=\{X\oplus X^{-1} : X\in \GL_b(k)\}$.
This makes the exact sequence of Theorem~\ref{thm:pseudo} collapse to
\begin{align*}
	1\to k^{\times} & \to (k^{\times} \times {\rm GL}_{b}(k))\times k^{\times}\to \Psi\Isom(\#)\to \Gal(k).
\end{align*}
Again we find $\GL_b(k)\times k^{\times}<\Psi\Isom(\#)=(\GL_b(k)\times k^{\times})\rtimes \Gal(k)$.

Finally if either $a$ or $c$ is $1$ then we obtain a nontrivial nilpotent radical $N$ in $\MM{\#}$, 
which is additively isomorphic to $k^a\otimes_k M_b(k)\otimes_k k^c$.  The involution given in 
Example~\ref{ex:adj-mat} confirms $1+N$ lies in $U(\LMR_{\#},\bar{\cdot})$.  Consequently the sequence of
Theorem~\ref{thm:pseudo} becomes
\begin{align*}
	1\to k^{\times} & \to (k^a\otimes_k M_b(k)\otimes_k k^c)\rtimes (k^{\times} \times {\rm GL}_{b}(k))\times k^{\times}\to \Psi\Isom(\#)\to \Gal(k).
\end{align*}
The result is our final claim.
\end{proof}

\subsection{Claims under degeneracy}\label{sec:degenerate}

The last general concern is that not all bimaps that arise in practice can be forced to be fully nondegenerate.
Given a bimap $*:U\times V\bmto W$ we can induce an associated fully nondegenerate bimap
$(*/\sqrt{*}): U/V^{\bot}\times V/U^{\top}\bmto U*V$ where
\begin{align*}
	 (u+V^{\top})(*/\sqrt{*})(v+U^{\bot}) & = u*v.
\end{align*}
For a subset $X$ of a module $Y$, define $C_{\GL(Y)}(X)=\{g\in \GL(Y) : xg=x\}$.

\begin{thm}\label{thm:aut-deg}
For a possibly degenerate $k$-bimap $*:U\times V\bmto W$ there an epimorphism 
$\Aut(*)\to \Aut(*/\sqrt{*})$ with kernel
\begingroup
\setlength{\thinmuskip}{0mu}
\begin{align}\label{eq:C}
	C & = C_{\GL_k(U)}(U/V^{\top})
		\times C_{\GL_k(V)}(V/U^{\bot})
			\times C_{\GL_k(W)}(U*V)
\end{align}
\endgroup
If the radicals each split in their respective groups (e.g. if $k$ is a field) then the extension 
of $C$ by $\Aut(*/\sqrt{*})$ is split.
\end{thm}
\begin{proof} Let $\phi\in \Aut(*)$.  It follows that $V^{\top} R_{\phi}^U\leq V^{\top}$, $U^\bot R_{\phi}^V\leq U^{\bot}$ and 
$(U*V)R_{\phi}^W\leq U*V$.  Therefore $\phi$ factors through $\hat{\phi}\in \Aut(U/V^{\top})\times \Aut(V/U^{\top})\times \Aut(U*V)$
and $\hat{\phi}\in \Aut((*/\sqrt{*})$.  This is a homomorphism with kernel living inside 
$C=C_{\GL_k(U)}(U/V^{\top})\times C_{\GL_k(V)}(V/U^{\bot})\times C_{\GL_k(W)}(U*V)$.  Suppose $\phi\in C$.  Then
for $u\in U$ and $v\in V$,
\begin{align*}
	uR_{\phi}^U*vR_\phi^V & = u*v=(u*v)R_{\phi}^W.
\end{align*}
So $\phi\in \Aut_k(*)$ and the proof is complete.
\end{proof}

%
\section{Applications to isomorphisms in algebra}\label{sec:archetype}

At last we circle back to claims of the introduction concerning automorphisms of filtered algebras and groups, and the example
in Corollary~\ref{coro:tri}.  We start by exploring an archetype of 
nilpotent algebras which resembles block upper triangular matrices but which is so varied that its variety has a dimension
equal to the dimension of the varieties of all associative and Lie algebras (and in the group case it is a logarithmically dense
subset of all groups).  These archetypes we denote by $T(*;S)$ and $U(*;S)$.  Then we reduce the
structure of $\Aut(T(*;S))$ and $\Aut(U(*;S))$ to $\Aut(*)$ and $\Psi\Isom(*)$ respectively
(Theorem~\ref{thm:Aut(T)}).  Then as a demonstration of Theorems~\ref{thm:auto} and \ref{thm:Skolem-Noether} 
we derive Corollary~\ref{coro:tri}.  

Many methods of this section are folklore that has arisen independently within the 
contexts of rings, algebras, and groups. They can also be interpreted as applications of more general 
correspondences that turn nilpotent algebraic structures into bilinear maps; see Bourbaki
for algebras \cite{Bourbaki:algebra}*{Chapter III\S 3}, and for groups see Higman's survey \cite{Higman:graded}.
Our explicit use of the archetype $\mathcal{T}(*;S)$ demonstrates that every possible bimap enters the picture
of isomorphisms in algebra.

\subsection{The archetypes $\mathcal{T}(*;S)$, $U(*;S)$, and $D(*;S)$}
For a $k$-bimap $*:U\times V\bmto W$ and a multiplicatively closed nonempty set $S\subseteq k$ we defined an algebraic structure:
\begin{align*}
	\mathcal{T}(*)=\mathcal{T}(*;S) & = \left\{ \begin{bmatrix} s & u & w \\ 0 & s & v \\ 0 & 0 & s \end{bmatrix} : 
		s\in S, u\in U, v\in V, w\in W \right\}.
\end{align*}
with multiplication understood as formal matrices, using $u*v$ wherever $u$ meets $v$. 
From here forward we use $S=k$ to identify the associative product $xy$ as formal matrices, 
$S=\{0\}$ to indentify the Lie product $[x,y]_+=xy-yx$, and $S=\{1\}$ to identify the group commutation product
$[x,y]_{\times}=x^{-1}y^{-1}xy$ which sits atop the group structure of the associative matrix product.  
For clarity we list these products in order for $\mathcal{T}(*;S)$.
\begin{align}
\begin{bmatrix} s & u & w \\ 0 & s & v \\ 0 & 0 & s \end{bmatrix}
	\begin{bmatrix} s' & u' & w' \\ 0 & s' & v' \\ 0 & 0 & s' \end{bmatrix}
	& =
	\begin{bmatrix} ss' & s'u+su' & s'w+u*v'+sw' \\ 0 & ss' & s'v+sv' \\ 0 & 0 & ss' \end{bmatrix}\label{eq:ass}\\
\left[\begin{bmatrix} 0 & u & w \\ 0 & 0 & v \\ 0 & 0 & 0 \end{bmatrix},
	\begin{bmatrix} 0 & u' & w' \\ 0 & 0 & v' \\ 0 & 0 & 0 \end{bmatrix}\right]_+
	& =
	\begin{bmatrix} 0 & 0 & u*v'-u'*v \\ 0 & 0 & 0 \\ 0 & 0 & 0 \end{bmatrix}\label{eq:comm}\\
\left[\begin{bmatrix} 1 & u & w \\ 0 & 1 & v \\ 0 & 0 & 1 \end{bmatrix},
	\begin{bmatrix} 1 & u' & w' \\ 0 & 1 & v' \\ 0 & 0 & 1 \end{bmatrix}\right]_{\times}
	& =
	\begin{bmatrix} 1 & 0 & u*v'-u'*v \\ 0 & 1 & 0 \\ 0 & 0 & 1 \end{bmatrix}\label{eq:comm-grp}
\end{align}
To write uniform proofs we use the the operations $\{\cdot,+,-,0\}$ for each of the above algebraic structures 
$\mathcal{T}=\mathcal{T}(*;S)$.  In this way $\mathcal{T}^2$ represents 
$[\mathcal{T}(*;\{0\}),\mathcal{T}(*;\{0\})]_+$ and $[\mathcal{T}(*;\{1\}),\mathcal{T}(*;\{1\})]_{\times}$.
Also, {\em ideal} is with respect to $\{\cdot,+,-,0\}$. So for example in a group $G$ an ideal $N$
satisfies $N\cdot G,G\cdot N\leq N$ if, and only if, $[N,G],[G,N]\leq N$; thus, ideal in the group
context agrees with normal.
\medskip

The archetypes  $\mathcal{T}(*;S)$ are one of several possible constructions.  One popular variation is to involve 
a $\tau$-weakly Hermitian product $*:U\times V\bmto W$ and define the subobject
\begin{align*}
	\mathcal{U}(*;S) & = \left\{ \begin{bmatrix} s & u & w \\ 0 & s & uR_{\tau}^U \\ 0 & 0 & s \end{bmatrix} : s\in S, u\in U, w\in W\right\}.
\end{align*}
This is especially popular when $*$ is alternating or symmetric.  In that case it leads
to (skew)-commutative associative rings and further groups and Lie algebras.  We can further vary the operators on the diagonal to be independent, or we can extend to
$(d\times d)$-matrices when we have a collection of distributive products $U_{ij} \times U_{jk}\bmto U_{ik}$, for $1\leq i<j\leq d$.
If we also include some knowledge of ${\rm Ext}_{\CC{}}(W, U\oplus V)$ we sample across even more of algebra.
A final archetype is use only bimaps where $U_*=V_*$ which are not required to be weakly Hermitian and then insist that $u=v$ on the subdiagonal of the matrices.  Those we denote by $D(*;S)$.
\medskip

Despite humble origins, the constructions such as $\mathcal{T}(*;S)$ occupy a substantial portion of the possible variability in common forms
of algebra, such as groups, and associative or Lie rings.  We can even count the variability.

Suppose $k$ is a commutative ring and $S\subseteq k$.  
Then for a $n=a+b+c$, the pairwise non-isomorphic 
$\mathcal{T}(*:k^a\times k^b\bmto k^c;S)$, is a variety of dimension $n^3/27+\Theta(n^2)$.
This is because we can choose $a\approx b\approx c\approx n/3$ and so we are choosing a tensor
from the tensor product space $k^{a}\otimes k^b\otimes k^c\cong k^{n^3/27}$.  Isomorphism does
not influence the number of tensors greatly because $\GL_a\times \GL_b\times \GL_c$ is an
an $O(n^2)$-dimensional algebraic group.
Likewise the variety of $\mathcal{U}(*:k^a\times k^a\bmto k^b;S)$ has dimension 
$2n^3/27+\Theta(n^2)$, for $n=a+b$.  

The range of options for our archetypes is remarkable because Neretin \cite{Neretin:enum}
shows that the variety of all commutative, resp. Lie, $k$-algebras of dimension $n$ is $2n^3/27+O(n^{3-\epsilon})$, 
for some $1\geq \epsilon > 0$ (presently $\epsilon=1/2$).  Identical bounds hold for finite groups 
\citelist{\cite{Higman:enum}\cite{Sims:enum}\cite{Pyber:enum}} and finite commutative rings \cite{Poonen:enum}.
For associative algebras the dimension of the variety jumps to $4n^3/27+O(n^{3-\epsilon})$ 
\citelist{\cite{Kruse-Price:enum}\cite{Neretin:enum}}.  Here
the archetype $D(*;S)$ determines a variety of dimension $4n^3/27+\Theta(n^2)$.  

It is fair to argue that the archetypes selected are limited in structure, such as nilpotence class $2$.  Nevertheless,
these are a substantial component of algebra and often a base case for inductions (nilpotence of class $1$
is abelian and largely unrelated to general nilpotence.)

\subsection{The automorphisms of the archetypes}
We now show how the objects $\mathcal{T}(*;S)$ depend on $*$.  Of course $*$ is
in their definition but we mean to recover $*$ from abstract properties of groups, rings, and algebras.
The importance of this step in considering isomorphism is that although the objects $\mathcal{T}(*;S)$ 
have been nicely represented, representations of an object do not in general carry over to representations
of their automorphism group.  Arguments for the archetypes $U(*;S)$ and $D(*;S)$ are largely unchanged.

Now we assume $N$ is the nilpotent radical of $\mathcal{T}:=\mathcal{T}(*;S)$, i.e. $N=\mathcal{T}(*;S)$ for $|S|=1$, and for $S=k$
$N=\mathcal{T}(*;\{0\})$ but treated as an ideal of the associative algebra $\mathcal{T}(*;S)$.  
Define  $\#:N/N^2\times N/N^2\bmto N^2$ by
\begin{align*}
	(x+N^2)\cdot (y+N^2) & = x\cdot y 
\end{align*}
The formulas of \eqref{eq:ass}--\eqref{eq:comm-grp} confirm $\circ$ is $k$-bilinear, $N^2\cong W$, and $N/N^2\cong U\oplus V$.
Indeed we can describe $\#:(U\oplus V)\times (U\oplus V)\bmto W$ by
\begin{align*}
	(u_1,v_1)\# (u_2,v_2) & = \left\{\begin{array}{cc}
		u_1*v_2, & S=k;\\
		 u_1*v_2-u_2*v_1, & S=\{0\} \textnormal{ or } S=\{1\}.
		\end{array}\right.
\end{align*}
We prove:

\begin{thm}\label{thm:Aut(T)}
Fix a field $k$. For each fully nondegenerate $k$-bimap $*:U\times V\bmto W$ the following holds.
\begin{align*}
	\Aut(\mathcal{T}(*;S)) & \cong \left\{
	\begin{array}{cc}
	 \hom_k(U\oplus V,W)\rtimes \Aut_k(*)\rtimes \Gal(k) & S=k;\\
	 \hom_k(U\oplus V,W)\rtimes \Psi\Isom_k(\#)\rtimes \Gal(k) & S=\{0\};\\
	 	 \hom_k(U\oplus V,W)\rtimes \Psi\Isom_k(\#) \Gal(k) & S=\{1\},2k=k.\\
	 \end{array}\right.	
\end{align*}
\end{thm}
\begin{proof}
Continue with the notation above.

Since both $N$ and $N^2$ are characteristic ideals of $\mathcal{T}$, 
as $k$-bimodules $N/N^2\cong U\oplus V$ and $N^2\cong W$.
So we can induce each automorphism $\alpha\in \Aut_k(\mathcal{T})$ as 
$(R_{\alpha}^{U\oplus V},R_{\alpha}^W)\in \Aut_k(U\oplus V)\times \Aut_k(W)$.  
As $x\alpha\cdot y\alpha=(x\cdot y)\alpha$ it follows that
\begin{align*}
	(u_1,v_1)R_{\alpha}^{U\oplus V} \# (u_2,v_2)R_{\alpha}^{U\oplus V} & = (u_1,v_1)\#(u_2,v_2) R_{\alpha}^W.
\end{align*}
This shows we should consider the subgroup ${^2\Aut}_k(*)=\{\phi=(f,g;h)\in \Aut_k(\#) : f=g\}$.
We have defined a group homomorphism
$\Phi:\Aut_k(\mathcal{T})\to {^2\Aut}_k(\#)\rtimes \Gal(k)$ where
\begin{align*}
	\alpha\Phi & = (R_{\alpha}^{U\oplus V},R_{\alpha}^{U\oplus V}; R_{\alpha}^W).
\end{align*}
We will prove $\Phi$ is surjective and split, so $\Aut(\mathcal{T})\cong \ker \Phi\rtimes({^2\Aut}_k(\#)\rtimes \Gal(k))$.

In the case of $S=k$, $\#$ is degenerate but full, and so
\begin{align*}
	\Aut(\#) & = \left\{\left(
		\begin{bmatrix} 
			R_{\phi}^U & \Lambda_2 \\
			0 &  \Lambda_1
		\end{bmatrix},
		\begin{bmatrix} 
			\Lambda_3 & 0 \\
			\Lambda_4 &  R_{\phi}^V
		\end{bmatrix}; 
		R_{\phi}^W \right) :  \begin{array}{cc} \phi\in \Aut(*)\\ \Lambda_1\in \Aut(V), & \Lambda_2:U\to V,\\
			\Lambda_3\in \Aut(U), & \Lambda_4:V\to U\end{array}  \right\};\\
{^2 \Aut}(\#) & = \left\{\left(
		\begin{bmatrix} 
			R_{\phi}^U & 0 \\
			0 &  R_{\phi}^V
		\end{bmatrix},
		\begin{bmatrix} 
			R_{\phi}^U & 0 \\
			0 &  R_{\phi}^V
		\end{bmatrix}; 
		R_{\phi}^W \right) :  \phi\in \Aut(*) \right\}\cong \Aut(*).
\end{align*}
Furthermore, for every $\phi\in \Aut(*)$,
\begin{align}\label{def:lift}
	\begin{bmatrix} s & u & w \\ 0 & s & v \\ 0 & 0 & s \end{bmatrix} & \mapsto 
	\begin{bmatrix} s & uR_{\phi}^U & wR_{\phi}^W \\ 0 & s & vR_{\phi}^V \\ 0 & 0 & s \end{bmatrix}.
\end{align}
defines an automorphism of $\mathcal{T}(*;S)$.  So $\iota:\Aut(*)\hookrightarrow \Aut(\mathcal{T})$ such
that $\iota\Phi$ is the identity.

On the other hand if $|S|=1$ then $\#$ is weakly-Hermitian with respect to $(1,1;-1)$ and 
${^2\Aut}_k(*)=\Psi\Isom_k(*)$.  Assuming $2k=k$, it follows that $\mathcal{T}(*;S)\cong \mathcal{U}(\frac{1}{2}\#;S)$,
via
\begin{align*}
	\begin{bmatrix} s & u & w \\ 0 & s & v \\ 0 & 0 & s \end{bmatrix}
		& \mapsto
			\begin{bmatrix} s & u\oplus v & w \\ 0 & s & u\oplus v \\ 0 & 0 & s \end{bmatrix}.		
\end{align*}
Evidently each $\phi\in \Psi\Isom(\#)$ lifts to an automorphism of $\mathcal{U}(\frac{1}{2}\#;S)$ as in \eqref{def:lift}.
It therefore also lifts to $\Aut(\mathcal{T}(*;S))$.
\smallskip

The kernel of $\Phi$ consists of automorphisms $\alpha$ which centralize $N/N^2$ and $N^2$.  As $N^3=0$ 
it follows that $(1-\alpha):X\mapsto (X-X\alpha)$ is a linear mapping with kernel containing $N^2$ and 
imaged contained in $N^2$.  That is, $\ker\Phi\hookrightarrow \hom(N/N^2,N^2)$.  Indeed, if
$\tau:N/N^2\to N^2$ is $k$-linear then define $1+\tau$ by
\begin{align*}
	(1+\tau):
	\begin{bmatrix}
	s & u & w \\ 0 & s & v \\ 0 & 0 & s 
	\end{bmatrix}
	\mapsto
	\begin{bmatrix}
	s & u & w+\tau(u\oplus v) \\ 0 & s & v \\ 0 & 0 & s 
	\end{bmatrix}.	
\end{align*}
This is a $k$-linear automorphism of $\mathcal{T}(*;S)$ and furthermore in the kernel of $\Phi$.
Therefore $\ker\Phi\cong \hom_k(U\oplus V,W)$.
\end{proof}

\subsection{$\Aut(\mathcal{T}_{a,b,c}(k;S))$ [Proof of Corollary~\ref{coro:tri}]}
Apply Theorem~\ref{thm:Aut(T)} to reduce the question to describing $\Aut(*)$ in the associative ring case, and $\Psi\Isom(\#)$ in
the Lie and group case.  Apply Theorems~\ref{thm:Skolem-Noether} and Theorem~\ref{thm:pseudo} respectively.
\hfill $\Box$

\subsection{General $\mathcal{T}(*;S)$}
The isomorphisms of $\mathcal{T}_{a,b,c}(k;S)$ leverage the complete understanding of
the bimap of $(a,b,c)$-matrix multiplication.  To study $\Aut(\mathcal{T}(*;S))$
we can still use matrices, only we expect an approximation not a complete picture.

We begin by recalling Figure~\ref{fig:scalar-diagram}.  In this diagram we relate a general bimap $*$ to tensors over
$\MM{*}$ and versors over $\LL{*}$ and $\RR{*}$ respectively.    Passing to the archetypes we see this correspondence
carried over into other categories such as rings, groups, and Lie algebras; see  Figure~\ref{fig:tensor-versor-triangle}.   The use of matrices is in fact honest
as we can represent tensor and versor products with coordinates in a $k$-vector space.
Therefore the archetypes such as $\mathcal{T}(\otimes_{\MM{*}};S)$ can be treated as
quotients of $\mathcal{T}_{a,b,c}(k;S)$, and similarly with versors.  But for
our purpose it simply helps us visualize the relationship of general archetypes
to those which come from matrices.

\smallskip
\begin{figure}[!htbp]
\begin{center}
\begin{tikzpicture}[node distance = 3cm]

\node (ZT) at (-3cm,1.5cm) {
	\begin{tikzpicture}[scale=0.10]
		\draw (0,21)  -- (21,0);
		\draw[step=1cm,black!25,very thin] (0,0) grid (21,21);		
		\fill[black] (13,21) rectangle (12,9);
		\fill[black!50] (21,9) rectangle (13,8);
		\fill[black!25] (21,21) rectangle (13,9);
		\node[fill=white] (ZLVt) at (6,6) {$\mathcal{T}(\otimes)$};
	\end{tikzpicture}
};

\node (MT) at (-3cm,-1.5cm) {
	\begin{tikzpicture}[scale=0.10]
		\draw (0,12)  -- (12,0);
		\draw[step=1cm,black!25,very thin] (0,0) grid (12,12);
		\fill[black] (8,12) rectangle (6,6);
		\fill[black!50] (12,6) rectangle (8,4);
		\fill[black!25] (12,12) rectangle (8,6);
		\node[fill=white] (MTt) at (5,-5) {$\mathcal{T}(\otimes_{\MM{*}})$};
	\end{tikzpicture}
};

\node (O) at (0cm,0cm) {
	\begin{tikzpicture}[scale=0.10]
		\draw[step=1cm,lightgray,very thin] (0,1) grid (20,2);
		\draw[step=1cm,lightgray,very thin] (4,0) grid (16,1);
		\fill[black] (0,1) rectangle (12,2);
		\fill[black!50] (12,1) rectangle (20,2);
		\fill[black!25] (4,0) rectangle (16,1);
		\node[fill=white] (ZLVt) at (10,6) {$\mathcal{T}(*)$};		
	\end{tikzpicture}
};

\node (LLV) at (1cm,3cm) {
	\begin{tikzpicture}[scale=0.10]
		\draw (0,11)  -- (11,0);
		\draw[step=1cm,lightgray,very thin] (0,0) grid (11,11);
		\fill[black] (7,11) rectangle (3,8);
		\fill[black!50] (11,8) rectangle (7,4);
		\fill[black!25] (11,11) rectangle (7,8);
		\node[fill=white] (ZLVt) at (5,15) {$\mathcal{T}({_{\LL{*}}\lversor})$};
	\end{tikzpicture}
};

\node (ZLV) at (3.5cm, 2cm) {
	\begin{tikzpicture}[scale=0.10]
		\draw (0,25)  -- (25,0);
		\draw[step=1cm,lightgray,very thin] (0,0) grid (25,25);
		\fill[black] (13,25) rectangle (1,24);
		\fill[black!50]  (25,24) rectangle (13,12);
		\fill[black!25] (25,25) rectangle (13,24);
		\node[fill=white] (ZLVt) at (7,7) {$\mathcal{T}(\lversor)$};
	\end{tikzpicture}
};

\node (RRV) at (1cm, -3cm) {
	\begin{tikzpicture}[scale=0.10]
		\draw (0,12)  -- (12,0);
		\draw[step=1cm,lightgray,very thin] (0,0) grid (12,12);
		\fill[black] (10,12) rectangle (6,6);
		\fill[black!50] (12,6) rectangle (10,2);
		\fill[black!25] (12,12) rectangle (10,6);
		\node[fill=white] (RRVt) at (-3,3) {$\mathcal{T}(\rversor_{\RR{*}})$};
	\end{tikzpicture}
};

\node (ZRV) at (3.5cm, -1cm) {
	\begin{tikzpicture}[scale=0.10]
		\draw (0,21)  -- (21,0);
		\draw[step=1cm,lightgray,very thin] (0,0) grid (21,21);
		\fill[black] (20,21) rectangle (12,9);
		\fill[black!50] (21,9) rectangle (20, 1);
		\fill[black!25] (21,21) rectangle (20,9);
		\node[fill=white] (ZRVt) at (6,6) {$\mathcal{T}(\rversor)$};
	\end{tikzpicture}
};
\draw[->] (ZT) -- (MT);
\draw[->] (ZT) -- (O);
\draw[->] (MT) -- (O);
\draw[->] (O) -- (LLV);
\draw[->] (O) -- (ZLV);
\draw[->] (LLV) -- (ZLV);
\draw[->] (O) -- (RRV);
\draw[->] (O) -- (ZRV);
\draw[->] (RRV) -- (ZRV);

\end{tikzpicture}
\end{center}
\caption{We trap  $\mathcal{T}(*)$ inside a triple of universal constructions.
Attaching rings $\LL{*}$, $\MM{*}$, $\RR{*}$ the triple contracts.  The tighter the constriction the more the 
surrounding members influence the properties of $\Aut(\mathcal{T}(*))$.
For $\mathcal{T}_{abc}(k)$ the triple contracts to the center and we obtain a perfect
understanding of $\Aut(\mathcal{T}_{abc}(k))$.}\label{fig:tensor-versor-triangle}
\end{figure}
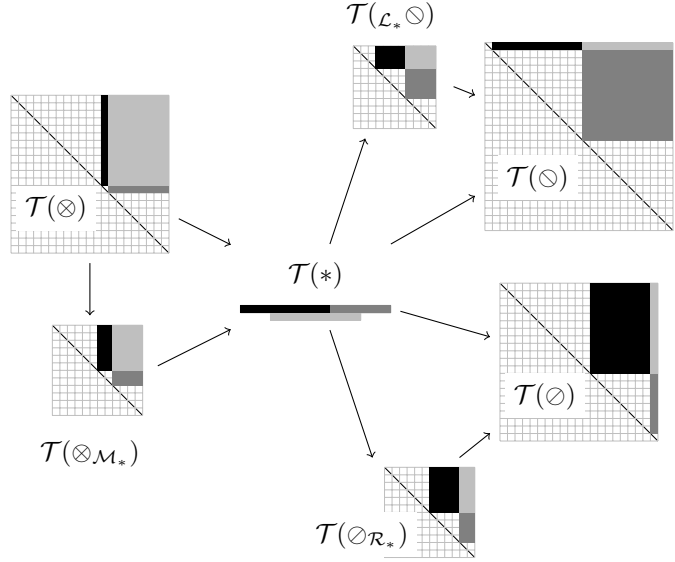

To understand the significance of Figure~\ref{fig:tensor-versor-triangle} we begin from 
the nilpotent quotient method.  The observation is that automorphisms and isomorphisms
of nilpotent groups, rings, and algebras can be approached inductively by treating them as quotients of relatively
free objects $F$ (relative to some variety such as nilpotent class $2$, then class 3, etc.).  If we can describe the automorphism
groups $\Aut(F)$, in the respective categories, and $A\cong F/N$, then $\Aut(A)$ can be recovered as
the stabilizer in $\Aut(F)$ of $N$.  Likewise isomorphism between quotients $F/N_1$ and $F/N_2$
is determined by whether or not $N_1$ and $N_2$ are in the same $\Aut(F)$-orbit.  The idea for such a
description of isomorphisms was seen early on in lectures of G. Higman \cite{Higman:chic}.  This
is the foundation of enumeration methods estimating the number of groups and algebras \citelist{\cite{Higman:enum}\cite{Kruse-Price:enum}\cite{Poonen:enum}\cite{Neretin:enum}} and
it is the leading method to compute automorphisms and isomorphisms  
\citelist{\cite{HNVL:NQ}\cite{Newman:NQ}\cite{OBrien}\cite{Eick:NQ}}.
\smallskip

We now observe the nilpotent quotient methods are just 1/3 of the whole picture.  Consider Figure~\ref{fig:tensor-versor-triangle}.
In the middle we place a group (or algebra) such as $\mathcal{T}(*)$.
The original nilpotent quotient method corresponds to lifting the problem to $\Aut(F)$, where we illustrate this
using $F=\mathcal{T}(\otimes)$.  Using tensor and versors we can construct two dual approaches {\em embedding} 
$\mathcal{T}(*)$ into $\mathcal{T}(\lversor)$ and $\mathcal{T}(\rversor)$.

The last important ingredient is to make use of rings $\LL{*}$, $\MM{*}$, and $\RR{*}$ that are designed around
the properties of $*$ (defined in Section~\ref{sec:prelims}).  For example, in matrix products $\M_{a\times b}(k)\times \M_{b\times c}\bmto\M_{a\times c}(k)$
we choose to tensor and versor with $\MM{*}=\M_b(k)$, $\LL{*}=\M_a(k)$, and $\RR{*}=\M_c(k)$.  As we show in 
Theorem~\ref{thm:universal}, each bimap $*$ has three rings associated to it in a universal sense so that
we can pass to smaller tensor and versor products.
We see this as the area shrinks in the corresponding formal matrices in Figure~\ref{fig:tensor-versor-triangle}.
This makes the nilpotent quotient/embedding methods aware of greater structure in the product of $\mathcal{T}(*)$.
Using the exact sequences of autotopism groups created here we can give specific descriptions of $\Aut(\otimes_{\MM{*}})$,
$\Aut(\lversor_{\LL{*}})$, and $\Aut(_{\RR{*}}\rversor)$.  Already the use of $\Aut(\otimes_{\MM{*}})$
was found in \cite{LW} to reduce the cost of isomorphism testing of quotients of $\mathcal{T}_{1,m,1}(p^e)$ 
from the prior cost of $O(p^{c(me)^2})$ to just $O((me)^6 \log^2 p)$ operations.  This was subsequently generalized
to a general method in \cite{BW:autotopism}.  The addition of versors opens these methods up to larger families of groups. 

\subsection{General automorphisms}\label{sec:radicals}
We close by looking at the implications on automorphisms of general algebraic objects.  We will tell the story for
groups but remark that it applies also to nonassociative rings and nonassociative loops by attaching
an associated graded algebra.  In the case of loops that was only recently seen to be possible in the
exciting work of Mostovoy \cite{Mostovoy}.

Following \cite{Wilson:alpha}, let  $M=\langle M,+,0,\prec\rangle$ denote a pre-ordered commutative monoid, e.g. $\mathbb{N}^c$ with
the point-wise or lexicographic partial order.
A {\em filter} $\phi:M \to 2^G$ on a group $G$ is a function into the subgroups such that
\begin{align*}
	(\forall & s,t\in M) & [\phi_s,\phi_t ] & \leq \phi_{s+t} & s\prec t & \Rightarrow \phi_s\geq \phi_t.
\end{align*}
The next theorem can be seen as a generalization of Theorem~\ref{thm:Aut(T)} to arbitrary filters.
In particular not only does it allow us to look at groups with general nilpotence class, it allows us to
insist that they are refined so that every associated bimap has semisimple rings $\LL{*}$, $\MM{*}$, 
and $\RR{*}$.  In particular this means we can invoke the Morita and Skolem-Noether theorems we have proved in 
Sections~\ref{sec:Morita}--\ref{sec:universal}.

\begin{thm}[\cite{Wilson:alpha}\cite{Wilson:Lie}]\label{thm:filter}
Fix a group $G$ with finite chain condition.
\begin{enumerate}[(i)]
\item
To every filter $\phi:M\to 2^G$  into the subgroups of $G$ there is a naturally induced $M$-graded
Lie algebra $L(\phi)=\bigoplus_{s\neq 0} L_s(\phi)$, $L_s(\phi)=\phi_s/\langle \phi_{s+t} : t\neq 0\rangle$. 
\item Every group has a filter $\phi$ into the characteristic subgroups of $G$ such that for every $s,t\neq 0$, 
the graded product $*:L_s(\phi)\times L_t(\phi)\bmto L_{s+t}(\phi)$ has
each of the rings $\LL{*}$, $\MM{*}$, $\RR{*}$, and $\CC{*}$ semisimple Artinian.
\item  If each $\phi_s$ is characteristic then there is a naturally induced homomorphism $\Aut(G)\to \Aut(L(\phi))$ whose kernel $K_{\phi}$ 
has a filter $\Delta\phi:M\to 2^G$ 
\begin{align*}
	\Delta\phi_s & = \{ f\in \Aut(G) : \forall t, [\phi_s, f]\leq \phi_{s+t}\}
\end{align*}
such that the $M$-graded Lie ring $L(\Delta\phi)$ is naturally represented in the graded derivation Lie ring $\Der L(\phi)$.
\end{enumerate}
\end{thm}

Theorem~\ref{thm:filter}(i) is proved in \cite{Wilson:alpha}*{Theorems 3.1}.  For part (ii) we use the process described in \cite{Wilson:alpha}*{Section 4}.
We move through the existing filter in some order, e.g. lexicographically, and compute the rings given above.  If we discover a nontrivial
Jacobson radical, we can refine the existing filter to remove that radical. The monoid $M$ is increased to $M\oplus \mathbb{N}$ and the process
begins all over.  At the end every subgroup in the filter is characteristic and each section $L_s$ is a semisimple module or each of the
rings described above.  Furthermore, the subgroups in the filter can be arranged into a series, though nothing in the results requires this.
The rings are very efficient to compute, see for example \cite{BW:slope}*{Theorem 4.1}. 
Part (iii) is proved in \cite{Wilson:Lie} and it shows how the filter continues to influence the structure of automorphism groups even for those automorphisms
which are represented trivially on the Lie structure.  In \cite{Wilson:Lie} we also
describe how to arrange for ascending variations on filter in a manner similar to how the lower and upper central series are related.  This inserts 
the study of autotopism groups into further contexts.

\section*{Acknowledgments}

This project started in earnest while spending an invigorating semester in The Einstein Institute for Mathematics, The Hebrew University.  I am grateful for the generous support of the institute and of
Aner Shalev, A. Mann, and A. Lubotzky.

\end{document}